\documentclass[12pt]{amsart}
\usepackage{url}
\usepackage{amssymb}
\usepackage{geometry}
\usepackage{verbatim}
\usepackage{multirow}
\newtheorem{theorem}{Theorem}
\newtheorem{remark}{Remark}
\newtheorem{lemma}{Lemma}

\newtheorem{proposition}{Proposition}

\newcommand{\tmop}[1]{\ensuremath{\operatorname{#1}}}
\numberwithin{equation}{section}
\numberwithin{lemma}{section}
\numberwithin{proposition}{section}
\numberwithin{corollary}{section}

\def\sumstar{\sideset{}{^*}\sum}
\def\sumfund{\sideset{}{^\flat}\sum}
\def\pamod{\! \! \! \! \pmod}
\title[Signs of Fourier coefficients]{Signs of Fourier coefficients of half-integral weight modular forms}
\author{Stephen Lester and Maksym Radziwi\l\l}
\address{School of Mathematical Sciences, Queen Mary University of London, 327 Mile End Road, London E1 4NS, UK}
\email{s.lester@qmul.ac.uk}
\address{Department of Mathematics, Caltech, 1200 E California BLVD, Pasadena, CA 91125, USA}
\email{maksym.radziwill@gmail.com}
\date{\today}
\begin{document}

\begin{abstract}
  Let $g$ be a Hecke cusp form of half-integral weight, level $4$ and belonging to Kohnen's plus subspace. Let $c(n)$ denote the $n$th Fourier coefficient of $g$, normalized so that $c(n)$ is real for all $n \geq 1$. A theorem of Waldspurger determines the magnitude of $c(n)$ at fundamental discriminants $n$ by establishing that the square of $c(n)$ is proportional to the central value of a certain $L$-function. The signs of the sequence $c(n)$ however remain mysterious.

  Conditionally on the Generalized Riemann Hypothesis, we show that $c(n) < 0$ and respectively $c(n) > 0$ holds for a positive proportion of fundamental discriminants $n$. Moreover we show that the sequence $\{c(n)\}$ where $n$ ranges over fundamental discriminants changes sign a positive proportion of the time.  
  
  Unconditionally, it is not known that a positive proportion of these coefficients are non-zero
and we prove results about the sign of $c(n)$ which are of the same quality as the best known non-vanishing results. Finally we discuss extensions of our result to general half-integral weight forms $g$ of level $4 N$ with $N$ odd, square-free. 
\end{abstract}



  %

\maketitle

\section{Introduction}

Let $k \ge 2 $ be an integer and $g$ be a weight $k+ \tfrac 12$ cusp form for $\Gamma_0(4)$. Every such $g$ has a Fourier expansion
\[
g(z)=\sum_{n \ge 1} c(n) n^{\frac{(k-1/2)}{2}}e(nz).
\]
The Fourier coefficients $c(n)$ encode arithmetic information. For instance under certain hypotheses, Waldspurger's Theorem shows that for fundamental discriminants $d$,  $|c(|d|)|^2$ is proportional to the central value of an $L$-function, so that the magnitude of the $L$-function essentially determines the size of the coefficient $c(n)$. However for $g$ with real Fourier coefficients, their signs remain mysterious.  In this article we contribute towards understanding the sign of such coefficients at fundamental discriminants through examining the number of coefficients which are positive (respectively negative) as well as the number of sign changes in-between. 




Signs of Fourier coefficients of half-integral weight forms have been studied by many authors following the works of Knopp-Kohnen-Pribitkin \cite{KKP} and Bruinier-Kohnen \cite{BK}, the former of which showed that such forms have infinitely many sign changes. Subsequent works \cite{HKKI, KLW, LRW} showed that the sequence $\{c(n)\}$ exhibits many sign changes under suitable conditions (such as the form having real Fourier coefficients). Notably, Jiang-Lau-L\"u-Royer-Wu \cite{JLLRW} showed for suitable $g$ 
that 
for every $\varepsilon > 0$ there are more than $\gg X^{2/9 - \varepsilon}$ sign changes along square-free integers $n \in [1, X]$.
They also showed 
this result can be improved assuming the Generalized Lindel\"of Hypothesis\footnote{Here the Generalized Lindel\"of Hypothesis is assumed for $L$-functions attached to quadratic twists of level 1 Hecke eigenforms.} with an exponent of $1/4$ in place of $2/9$.

For an integral weight Hecke cusp form $f$ the second named author and Matom\"aki \cite{MatomakiRadziwill}  proved a stronger result, establishing a positive proportion of sign changes along the positive integers. This uses the multiplicativity of the Fourier coefficients of $f$ in a fundamental way. Fourier coefficients of half-integral weight Hecke cusp forms lack this property, except at squares. So one may wonder whether the Fourier coefficients of half-integral weight Hecke cusp forms also exhibit a positive proportion of sign changes, along the sequence of fundamental discriminants.

In this article we answer this question in the affirmative. We show that under the assumption of the Generalized Riemann Hypothesis (GRH) there exists a positive proportion of fundamental discriminants at which the Fourier coefficients of a suitable half-integral weight form are positive as well as a positive proportion at which the coefficients are negative. Moreover, we show under GRH that the coefficients exhibit a positive proportion of sign changes along the sequence of fundamental discriminants. 

For simplicity, our results are stated for the Kohnen space $S_{k+1/2}^+$, which consists of all weight $k+1/2$ modular forms on $\Gamma_0(4)$ whose $nth$ Fourier coefficient equals zero whenever $(-1)^k n \equiv 2,3 \pmod{4}$. In this space Shimura's correspondence between half-integral weight forms and integral weight forms is well understood. Kohnen proved \cite{KohnenMathAnn} there exists a Hecke algebra isomorphism between $S_{k+1/2}^+$ and the space of level $1$ cusp forms of weight $2k$.
Also, every Hecke\footnote{We call a weight $k+1/2$ cusp form on $\Gamma_0(4)$ a Hecke cusp form if it is an eigenfunction of the Hecke operator $T_{p^2}$ (see \cite{Shimura1}) for each $p>2$.}  cusp form $g \in S_{k+1/2}^+$ can be normalized so that it has real coefficients\footnote{The numbers $c(n)n^{(k-1/2)/2}$ lie in the field generated over $\mathbb Q$ by the Fourier coefficients of its Shimura lift, which is a level $1$ Hecke eigenform of weight $2k$, so these numbers are real and algebraic (see Proposition 4.2 of \cite{KP} and also the remarks before Theorem 1 of \cite{KZ}).} and from here on we assume that $g$ has been normalized in this way.

Let $\mathbb{N}^{\flat}$ denote the set of fundamental discriminants of the form $8n $ with $n > 0$ odd, square-free. Also, let $\mathbb N_g^{\flat}(X)=\{n \in \mathbb{N}^{\flat} \cap [1,X] : c(n) \neq 0 \}$.

\begin{theorem} \label{thm:sign} Assume the Generalized Riemann Hypothesis. Let $k \ge 2$ be an integer and $g \in S_{k+1/2}^+$ be a Hecke cusp form. 
Then for all $X$ sufficiently large the number of sign changes of the sequence $\{ c(n) \}_{n \in \mathbb N_g^{\flat}(X)}$ is $ \gg X$.

  In particular, for all $X$ sufficiently large we can find $\gg X$ integers $d \in \mathbb{N}^{\flat} \cap [1,X]$ such that $c(d) < 0$, and we can find $\gg X$ integers $d \in \mathbb{N}^{\flat}\cap [1,X]$ such that $c(d) < 0$. 
\end{theorem}

The proof of Theorem \ref{thm:sign} uses the explicit form of Waldspurger's Theorem due to Kohnen and Zagier \cite{KZ}. Given a Hecke cusp form $g \in S_{k+1/2}^+$ they show that for each fundamental discriminant $d$ with $(-1)^k d>0$ that
\begin{equation} \label{eq:waldspurger}
    |c(|d|)|^2=  L(\tfrac12, f \otimes \chi_d) \cdot \frac{(k-1)!}{\pi^k } \cdot \frac{\langle g,g \rangle}{\langle f,f \rangle}.
\end{equation}
Here $f$ is a weight $2k$ Hecke cusp form of level $1$ which corresponds to $g$ as described above, $L(s,f\otimes \chi_d)$ is the $L$-function
\[
L(s,f\otimes \chi_d)=\sum_{n \ge 1} \frac{\lambda_f(n) \chi_d(n)}{n^s}, \qquad \tmop{Re}(s)>1
\]
where $\lambda_f(n)$ are the Hecke eigenvalues of $f$, $\chi_d(\cdot)$ is the Kronecker symbol. Also,
\[
\langle f,f \rangle= \int_{\tmop{SL}_2(\mathbb Z) \backslash \mathbb H} y^{2k}|f(z)|^2 \, \frac{dx dy}{y^2}, \qquad \qquad \langle g,g \rangle= \frac16 \int_{\Gamma_0(4) \backslash \mathbb H} y^{k+1/2}|g(z)|^2 \, \frac{dx dy}{y^2}.
\]
In Theorem \ref{thm:sign} we assume GRH for $L(s, f \otimes \chi_{d})$ for every fundamental discriminant $d$ with $(-1)^k d>0$.

The restriction to $d \in \mathbb{N}^{\flat}$ is important. As the proof of Theorem \ref{thm:uncondsign} below will show, it is easy to produce many positive (resp. negative) coefficients along \textit{the integers}, assuming a suitable non-vanishing result. 


\subsection{Unconditional results}

We also are able to prove a quantitatively weaker yet unconditional result.

\begin{theorem} \label{thm:sign2}
  Let $k \ge 2$ be an integer and $g \in S_{k+1/2}^+$ be a Hecke cusp form. Then for any $\varepsilon>0$ and all $X$ sufficiently large
  the sequence $\{ c(n) \}_{n \in \mathbb N_g^{\flat}(X)}$ has $\gg X^{1-\varepsilon}$ sign changes.
  
\end{theorem}

Theorems \ref{thm:sign} and \ref{thm:sign2} quantitatively match the best known non-vanishing results for Fourier coefficients of half-integral weight forms that are proved using analytic techniques, conditionally under GRH and unconditionally (resp.). In particular, Theorem \ref{thm:sign} gives a different proof that a positive proportion of these coefficients are non-zero, for Hecke forms in the Kohnen space. It should be noted that Ono-Skinner \cite{OS} have shown that there exist $\gg X/\log X$ fundamental discriminants at which these coefficients are non-zero for such forms.\footnote{As discussed above for forms in the Kohnen space the Fourier coefficients are algebraic integers which lie in a number field \cite[Proposition 4.2]{KP} so that the Fundamental Lemma of \cite{OS} applies.} However, their result does not give a quantitative lower bound on the size of the coefficients, whereas the analytic estimates do provide such information, which is crucial for our argument.



 On the other hand, 
 using the result of Ono-Skinner it is not difficult to produce both many positive Fourier coefficients and many negative ones, at \textit{integers}.


 \begin{theorem} \label{thm:uncondsign}
 Let $k \ge 2$ be an integer and $g \in S_{k+1/2}^+$ be a Hecke cusp form.
Then, for all sufficiently large $X$
\[
\# \{ n \le X : c(n)>0 \} \gg \frac{X}{\log X}
\quad \text{and} \quad
\# \{ n \le X : c(n)<0 \} \gg  \frac{X}{\log X}.
\]
 \end{theorem}

 \subsection{Extensions beyond the Kohnen plus space}

Since by now Shimura's correspondence is fairly well-understood (see \cite{BP}) we show in the Appendix that the conclusion of
 Theorems \ref{thm:sign} and \ref{thm:sign2} holds for every half-integral weight Hecke cusp form on $\Gamma_0(4)$\footnote{For $g \notin S_{k+1/2}^+$ it is possible that $c(8n)\mu^2(2n)$ is zero for each $n \in \mathbb N$, in this case we detect sign changes of $\{c(n)\}$ where  $n$ ranges over the set $\{ n \le X : n \text{ is even and }  \mu^2(n)c(n) \neq 0 \}$.}.
 Additionally, we can prove analogs of Theorems \ref{thm:sign} and \ref{thm:sign2} which hold for weight $k+\tfrac12$  ($k \ge 2$) cusp forms $g$ of level $4N$ with $N$ odd and square-free provided that $g$ corresponds (through Shimura's correspondence) to an integral weight newform. The necessary modifications to our argument and precise statements of the results are in the Appendix.





\subsection{Numerical examples} 

To illustrate our results above with a concrete example consider the following weight $\frac{13}{2}$ Hecke cusp form 
\[
\delta(z)=\frac{60}{2 \pi i} \left(2 G_4(4z)\theta'(z)-G_4'(4z)\theta(z) \right)=\sum_{\substack{n \ge 1 \\n \equiv 0,1 \pamod{4}}} \alpha_{\delta}(n) e(nz),
\]
where
\[
G_k(z)=\frac{1}{2} \cdot \zeta(1 - k) + \sum_{n \ge 1} \sigma_{k-1}(n)e(nz) \quad \text{ and} \quad \theta(z)=\sum_{n \in \mathbb Z} e(n^2z).
\]

The modular form $\delta(z)$ corresponds to the modular discriminant $\Delta(z)$ under the Shimura lift.
Assuming GRH for $L(\tfrac12, \Delta \otimes \chi_d)$ for every fundamental discriminant $d$, Theorem \ref{thm:sign} implies that there is a positive proportion of sign changes of $\alpha_{\delta}(n)$ along the sequence of fundamental discriminants of the form $8d$.
In fact numerical evidence below suggests that the Fourier coefficients of $\delta(z)$ change sign approximately one half of the time. Given a subset $\mathcal{L}$, let $\mathcal{S}_{\delta, \mathcal{L}}(X)$ denote the number of sign changes of $\alpha_{\delta}(n)$ along $n \in \mathcal{L} \cap [1, X]$, and denote by $\mathcal{N}_{\delta, \mathcal{L}}(X)$ the cardinality of $\mathcal{L} \cap [1, X]$. We then have the following numerical data. 

\vspace{.1 in}
\begin{tabular}{ |p{2.4cm}||p{2.1cm}|p{2.1cm}|p{2.1cm}|p{2.1cm}|p{2.1cm}|  }
 \hline
$X$ & $2\times 10^5$ & $2 \times 10^6$ & $2 \times 10^7$ & $2 \times 10^8 $ & $2 \times 10^9 $ \\
 \hline
 $\mathcal S_{\delta,\mathbb N}(X)$   &  $50291$&   $501 163$ & $5000867$ & $50000368$ & $500027782$ \\ 
 $ \frac{1}{\mathcal N_{\delta}(X) } \mathcal S_{\delta, \mathbb N}(X)$ &  $0.502915\ldots$ & $0.501163\ldots$ & $0.500086\ldots$ & $0.500003\ldots$ & $0.500027\ldots$\\
 \hline
\end{tabular}
\vspace{.1 in}

If we restrict to $\mathbb{N}^{\flat} = \{ 8 d : d \text{ odd, square-free}\}$ then we find the following data. 

\vspace{.1 in}
\begin{tabular}{ |p{2.4cm}||p{2.1cm}|p{2.1cm}|p{2.1cm}|p{2.1cm}|p{2.1cm}|p{2.1cm}|  }
 \hline
$X$ & $2\times 10^5$ & $2 \times 10^6$ & $2 \times 10^7$ & $2 \times 10^8$ & $2 \times 10^9$\\
 \hline
 $\mathcal S_{\delta, \mathbb N^{\flat}}(X)$   & $5049$&   $50734$ & $506589$ & $5065686$ & $50663938$\\ 
 $ \frac{1}{\mathcal N^{\flat}(X) } \mathcal S_{\delta,\mathbb N^{\flat}}(X)$ &  $0.498223\ldots$ & $0.500740\ldots$ & $0.499980\ldots$ & $0.499963\ldots$ & $0.500033\ldots$ \\
 \hline
\end{tabular}
\vspace{.1 in}


\subsection{Main Estimates} 


The main results follow from the following three propositions. The first two of the propositions allow us to control the size of $c(\cdot)$ by introducing a mollifier $M(\cdot;\cdot)$ which is defined in \eqref{eq:mollifierdef}. In our application the specific shape of the mollifier is crucial to the success of the method.
We have constructed this mollifier to counteract the large values of
\[
\exp\left(\sum_{p \le |d|^{\varepsilon}} \frac{\lambda_f(p)}{\sqrt{p}} \chi_d(p) \right)
\]
that contribute to the bulk of the moments of $L(\tfrac12, f\otimes \chi_d)$. Essentially, we are mollifying the $L$-function through an Euler product, as opposed to traditional methods which use a Dirichlet series. This approach was sparked by innovations in understanding of the moments of $L$-functions, such as the works of Soundararajan \cite{SoundMoments}, Harper \cite{Harper}, and Radziwi\l \l \,-Soundararajan \cite{RadziwillSound}.

\begin{proposition} \label{prop:secondmoment}
Assume GRH. Let $g \in S_{k+1/2}^+$ be a Hecke cusp form. Also, let $M(\cdot;\cdot)$ be as defined in \eqref{eq:mollifierdef}. Then
\[
\sum_{\substack{ n \le X \\ 2n \text{ is } \square-\text{free}}} |c(8n)|^4 M( (-1)^k 8n; \tfrac 12 )^4 \ll X.
\]
\end{proposition}

\begin{proposition} \label{prop:firstmoment}
Let $g \in S_{k+1/2}^+$ be a Hecke cusp form. Also, let $M(\cdot;\cdot)$ be as defined in \eqref{eq:mollifierdef}. Then
\[
\sum_{\substack{n \le X \\ 2n \text{ is } \square-\text{free}}} |c(8n)|^2 M((-1)^k8 n; \tfrac{1}{2})^2 \asymp X.
\]
\end{proposition}

The other key ingredient of our results is an estimate for sums of Fourier coefficients summed against a short Dirichlet polynomial over short intervals. This is proved through estimates for shifted convolution sums of half-integral weight forms.
 
\begin{proposition} \label{mainprop}
Let $g$ be a cusp form of weight $k+\tfrac12$ on $\Gamma_0(4)$.
  Let $(\beta(n))_{n \geq 1}$ be complex coefficients such that $|\beta(n)| \ll_{\varepsilon} n^{\frac12+\varepsilon}$ for all $\varepsilon > 0$. 
  Let
  $$
  M((-1)^k n) := \sum_{m \leq M} \frac{\beta(m) \chi_{(-1)^k n}(m)}{\sqrt{m}}
  $$
  where $\chi_{(-1)^k n}(m)$ denotes the Kronecker symbol. 
  Uniformly for $1 \leq y, M \leq X^{\frac{1}{1092}}$
  \begin{align*}
  \sum_{X \leq x \leq 2X} \Big | & \sum_{\substack{x \leq 8n \leq x + y \\ 2n \text{ is } \square-\text{free}}} c(8n) M((-1)^k 8 n) \Big | \\ & \ll_{k} X \sqrt{y} \cdot \Big ( \frac{1}{X^{k + \frac 12}} \sum_{\substack{2n \text{ is } \square-\text{free}}} |c(8 n) M((-1)^k 8 n)|^2 \cdot n^{k - \frac 12} e^{- \frac{4\pi n}{X}} \Big )^{1/2} + X^{1 - \frac{1}{2184} + \varepsilon}.
  \end{align*}
\end{proposition}

In particular, Proposition \ref{mainprop} holds for $M(\cdot)=M(\cdot;\tfrac12)$ as defined in \eqref{eq:mollifierdef}.

\section{The Proof of Theorem \ref{thm:sign}}

The basic method of proof follows a straightforward approach. Observe that since $M((-1)^k n; \tfrac{1}{2})>0$ (see the discussion before and after \eqref{eq:mollifierdef}), if 
\[
\bigg|\sum_{\substack{x \le 8n \le x+y\\ 2n \text{  is  } \square-\text{free}}} c(8n) M((-1)^k 8n; \tfrac{1}{2}) \bigg| < \sum_{\substack{x \le 8n \le x+y\\ 2n \text{  is  } \square-\text{free}}}
|c(8n)| M((-1)^k 8n; \tfrac{1}{2})
\]
then the sequence $\{ \mu^2(2n)c(8n) \}$ must have at least one sign change in the interval $[x, x+y]$. To analyze the sums above we use a direct approach, which was developed in \cite{MatomakiRadziwill}, where sign changes of integral weight forms was studied. The key input is Proposition \ref{mainprop}.

\begin{lemma} \label{lem:upperbdmt} 
Let $\Lambda\ge 1$.
There exists $\delta>0$ such that for $2 \le y \le X^{\delta}$ we have for all but at most $\ll X/\Lambda$
integers $X \le x \le 2X$ that
\[
\bigg| \sum_{\substack{ x \le 8n \le x+y \\ 2n \text{ is } \square\text{-free}}} c(8n) M((-1)^k 8n; \tfrac{1}{2}) \bigg| \le \Lambda \sqrt{y}.
\]
\end{lemma}
\begin{proof}
This follows from Markov's inequality combined with Proposition \ref{prop:firstmoment} and Proposition \ref{mainprop}.\end{proof}

\begin{lemma} \label{lem:lowerbdmt} Let $2 \le y \le X/2$. Then
for all sufficiently small $\varepsilon>0$ 
there exists a subset of integers $X\le x \le 2X$ which contains   $ \gg \varepsilon X$ integers
such that
\[
\sum_{ \substack{x \le 8n \le x+y \\  2n \text{ is } \square\text{-free} }} |c(8n)|M((-1)^k 8n; \tfrac{1}{2}) > \varepsilon y.
\]
\end{lemma}


\begin{proof}
For sake of brevity write $C(n)=|c(8n)|M((-1)^k8 n; \tfrac{1}{2})$.
By H\"older's inequality
\begin{equation} \label{eq:holder}
 \sum_{ \substack{ X \le 8n \le 2X \\  2n \text{ is } \square\text{-free}  } } C(n)^2 \le \Bigg(  \sum_{\substack{ X \le 8n \le 2X \\ 2n \text{ is } \square\text{-free}   }} C(n)  \Bigg)^{2/3} \Bigg(  \sum_{ \substack{X \le 8n \le 2X  \\ 2n \text{ is } \square\text{-free} }} C(n)^4  \Bigg)^{1/3}.
\end{equation}
Applying Proposition \ref{prop:firstmoment} it follows that the LHS is $\gg X$. 
 Applying Proposition \ref{prop:secondmoment} the second sum on the RHS is $\ll X$. Hence, we conclude that
\[
\sum_{\substack{ X \le 8n \le 2X \\ 2n \text{ is } \square\text{-free} }} C(n) \gg X.
\]
Also, by Proposition \ref{prop:firstmoment} we have
\[
\sum_{\substack{ X \le 8n \le 2X \\ 2n \text{ is } \square\text{-free} \\ C(n) > 1/\varepsilon}} C(n) < \varepsilon  \sum_{\substack{ X \le 8n \le 2X \\ 2n \text{ is } \square\text{-free} }} C(n)^2 \ll \varepsilon X.
\]
So that for $\varepsilon$ sufficiently small
\begin{equation} \label{eq:keybd}
\sum_{\substack{ X \le 8n \le 2X \\ 2n \text{ is } \square\text{-free} \\  C(n) \le 1/\varepsilon}} C(n) \gg X.
\end{equation}

Let $U$ denote the subset of integers $X \le x \le 2X$ such that
\[
\sum_{\substack{x \le 8n \le x+y \\   2n \text{ is } \square\text{-free} \\ C(n) \le 1/ \varepsilon  }} C(n) \le \varepsilon y .
\]
Using \eqref{eq:keybd} we get that
\[
\begin{split}
X \ll \sum_{\substack{ X+y \le 8n \le 2X \\ 2n \text{ is } \square\text{-free} \\   C(n)  \le  1/\varepsilon}} C(n) & \le \frac{1}{y} \sum_{X \le x \le 2X} \sum_{\substack{ x \le 8n \le x+y \\ 2n \text{ is } \square\text{-free} \\  C(n) \le 1/\varepsilon}} C(n)  \\
&= \frac{1}{y} \bigg(\sum_{x \in U} +\sum_{ \substack{X \le x \le 2X \\ x \notin U}} \bigg) \sum_{\substack{ x \le 8n \le x+y \\ 2n \text{ is } \square\text{-free} \\  C(n) \le 1/\varepsilon}} C(n)  \\
& \ll \frac{1}{y} \cdot \bigg( (\varepsilon y)  X + \frac{y}{\varepsilon} \cdot \sum_{ \substack{X \le x \le 2X \\ x \notin U}}  1 \bigg) .
\end{split}
\]
Since the first term above is of size $\varepsilon X$ we must have that the second term is $\gg X$ since $\varepsilon$ is sufficiently small. So $\# \{ X \le x \le 2X : x \notin U\} \gg \varepsilon X$.\end{proof}

\begin{proof}[Proof of Theorem \ref{thm:sign}]

In Lemma \ref{lem:upperbdmt} take $\Lambda=\frac{1}{\varepsilon^2}, y=\frac{1}{\varepsilon^6}$ so for except at most $\ll \varepsilon^2 X$ integers $X \le x \le 2X$ 
we have that
\begin{equation}
    \label{eq:upperbd}
 \bigg|  \sum_{\substack{ x \le 8n \le x+y \\ 2n \text{ is } \square\text{-free}}} c(8n) M((-1)^k 8n; \tfrac{1}{2}) \bigg| \le \Lambda \sqrt{y} =\frac{1}{\varepsilon^5}.
\end{equation}
 In Lemma \ref{lem:lowerbdmt} take $ y=\frac{1}{\varepsilon^6}$ so there are $\gg \varepsilon X$ integers $X \le x \le 2X$ such that
\begin{equation}
    \label{eq:lowerbd}
 \sum_{\substack{ x \le 8n \le x+y \\ 2n \text{ is } \square\text{-free}}} |c(8n)| M((-1)^k 8n; \tfrac{1}{2})> \varepsilon y=\frac{1}{\varepsilon^5}.
\end{equation}
Combining \eqref{eq:upperbd} and \eqref{eq:lowerbd} we get that there are $\gg \varepsilon X$ integers $X \le x \le 2X$ such  that  
\begin{equation} \notag
 \sum_{\substack{ x \le 8n \le x+y \\ 2n \text{ is } \square\text{-free}}} |c(8n)| M((-1)^k 8n; \tfrac{1}{2}) > \bigg|  \sum_{\substack{ x \le 8n \le x+y \\ 2n \text{ is } \square\text{-free}}} c(8n) M((-1)^k 8n; \tfrac{1}{2}) \bigg|.
\end{equation}
Since $M>0$ (see the discussion after \eqref{eq:mollifierdef}) this implies there exists at least $\gg \varepsilon X$ integers $X \le x \le 2X$ such that $[x,x+y]$ contains a sign change of the sequence $\{\mu^2(2n)c(8n)\}$. Since every sign change of $\{\mu^2(2n)c(8n)\}$ on $[X, 2X]$ yields at most $\lfloor y \rfloor +1$ intervals $[x,x+y]$ which contain a sign change it follows that there are at least $ \gg \varepsilon \frac{X}{y}=\varepsilon^7 X$ sign changes in $[X,2X]$, which completes the proof of Theorem \ref{thm:sign}.
\end{proof}


\section{The Proofs of Theorem \ref{thm:sign2} and Theorem \ref{thm:uncondsign}}
\subsection{The proof of Theorem \ref{thm:sign2}}
 Throughout we will need the following estimate.
\begin{lemma} \label{le:elementary}
We have
  $$
  \sum_{\substack{X \leq n \leq 2X \\ 2n \text{ is } \square\text{-free}}} |c(8n)|^2 \asymp X. 
   $$
 \end{lemma}
 \begin{proof}
 This follows immediately by applying \eqref{eq:waldspurger} along with  Proposition \ref{prop:twisted} below with $u=1$.
 \end{proof}

We are now ready to start the preparations for the proof of Theorem \ref{thm:sign2}, which as it turns out is considerably easier to prove than Theorem \ref{thm:sign}. 

\begin{lemma} \label{le:smallvalues} Let $\Lambda \ge 1$.
  There exists $\delta > 0$ such that for all $2 \leq y \leq X^{\delta}$ we have for all but at most $\ll X/\Lambda$ integers $X \leq x \leq 2X$ that
  $$
  \Big | \sum_{\substack{x \leq 8n \leq x + y \\ 2n \text{ is } \square\text{-free}}} c(8n) \Big | \leq \Lambda \sqrt{y}. 
  $$
\end{lemma}
\begin{proof}
  This follows from Markov's inequality combined with Proposition \ref{mainprop}  (with the choice $\beta(1) = 1$ and $\beta(m) = 0$ for all $m \geq 2$)  and Lemma \ref{le:elementary}. 
\end{proof}

\begin{lemma} \label{le:largevalues}
Let $\varepsilon >0$ and $2 \le y \le X/2$.  Then there exist $\gg X^{1-\frac32 \varepsilon}$ integers $X \le x \le 2X$ such that
  $$
  \sum_{\substack{x \leq 8n \leq x + y \\ 2n \text{ is } \square\text{-free}}} |c(8 n)| >  \frac{y}{X^{\varepsilon}}.
  $$
\end{lemma}
\begin{proof}
Note that using \eqref{eq:waldspurger} and Heath-Brown's result \cite[Theorem 2]{HB} we get that
\[
\sum_{\substack{X \le 8n \le 2X \\  2n \text{ is } \square\text{-free}}} |c(8n)|^4 \ll X^{1+\varepsilon}.
\]
Hence, using the above estimate along with Lemma \ref{le:elementary} we can apply H\"older's inequality as in \eqref{eq:holder} to get
\begin{equation} \label{eq:holder2}
\sum_{\substack{X \le n \le 2X \\  2n \text{ is } \square\text{-free}}} |c(8n)| \gg X^{1-\varepsilon/2}.
\end{equation}
Also, using Lemma \ref{le:elementary} we have
\[
\sum_{\substack{X \le 8n \le 2X \\  2n \text{ is } \square\text{-free} \\  |c(8n)| > X^{\varepsilon}}} |c(8n)| < X^{-\varepsilon} \sum_{\substack{X \le 8n \le 2X \\  2n \text{ is } \square\text{-free}}} |c(8n)|^2 \ll X^{1-\varepsilon}.
\]
So combining this along with \eqref{eq:holder2} we get
\begin{equation} \label{eq:keybd2}
\sum_{\substack{X \le 8n \le 2X \\  2n \text{ is } \square\text{-free} \\  |c(8n)| \le X^{\varepsilon}}} |c(8n)| \gg X^{1-\varepsilon/2}.
\end{equation}

Let $U$ denote the subset of integers $X \le x \le 2X$ such that
\[
\sum_{\substack{x \le 8n \le x+y \\   2n \text{ is } \square\text{-free} \\ |c(8n)|  \le  X^{\varepsilon} }} |c(8n)| \le \frac{y}{X^{\varepsilon}} .
\]
Using \eqref{eq:keybd2} we get that
\[
\begin{split}
X^{1-\varepsilon/2} \ll \sum_{\substack{ X+y \le 8n \le 2X \\ 2n \text{ is } \square\text{-free} \\ |c(8n)|  \le X^{\varepsilon}}} |c(8n)| & \le \frac{1}{y} \sum_{X \le x \le 2X} \sum_{\substack{ x \le 8n \le x+y \\ 2n \text{ is } \square\text{-free} \\   |c(8n)| \le X^{\varepsilon}}} |c(8n)|  \\
&= \frac{1}{y} \bigg(\sum_{x \in U} +\sum_{ \substack{X \le x \le 2X \\ x \notin U}} \bigg) \sum_{\substack{ x \le 8n \le x+y \\ 2n \text{ is } \square\text{-free} \\   |c(8n)| \le X^{\varepsilon}}} |c(8n)|  \\
& \ll \frac{1}{y} \cdot \bigg( \frac{y}{X^{\varepsilon}} X + y X^{\varepsilon} \cdot \sum_{ \substack{X \le x \le 2X \\ x \notin U}}  1 \bigg) .
\end{split}
\]
Since the first term is of size $X^{1-\varepsilon}$ we must have that the second term is $\gg X^{1-\varepsilon/2}$ so that $\# \{ X \le x \le 2X : x \notin U\} \gg  X^{1-\frac32 \varepsilon}$. \end{proof}

We are now ready to prove Theorem \ref{thm:sign2}.
\begin{proof}[Proof of Theorem \ref{thm:sign2}]
  Let $y = X^{6 \varepsilon}$ and $\Lambda=\frac{y^{1/2}}{X^{\varepsilon}}$. By Lemma \ref{le:smallvalues} we have
  $$
  \Big | \sum_{\substack{x \leq 8n \leq x + y \\ 2n \text{ is } \square\text{-free}}} c(8 n) \Big | \leq \frac{y}{X^{\varepsilon}}
  $$
  for all $X \leq x \leq 2X$ with at most $X^{1+\varepsilon} y^{-1/2} = X^{1 - 2 \varepsilon}$ exceptions. On the other hand, by Lemma \ref{le:largevalues} we have
  $$
  \sum_{\substack{x \leq 8n \leq x + y \\ 2n \text{ is } \square\text{-free}}} |c(8 n)| > \frac{y}{X^{\varepsilon}}
  $$
  on a subset of cardinality at least $X^{1 - 3\varepsilon/2}$. Therefore the two subsets intersect on at least $\gg X^{1 - 3 \varepsilon/2}$ values of $x$, and therefore give rise to at least $X^{1 - 15 \varepsilon/2}$ sign changes in $[X,2X]$.  
  \end{proof}

  The proof of Theorem \ref{thm:uncondsign} is completely elementary and does not depend on any of the other techniques developed here.

  \begin{proof} [Proof of Theorem \ref{thm:uncondsign}]

Let $g \in S_{k+1/2}^+$ be a Hecke cusp form and $f$ denote the weight $2k$ Hecke cusp form of level $1$ that corresponds to $g$. For $d$ a fundamental discriminant with $(-1)^k d>0$
 \begin{equation} \label{eq:shimura}
 c(n^2|d|)=c(|d|) \sum_{r|n} \frac{\mu(r) \chi_d(r) }{\sqrt{r}}\lambda_f\left( \frac{n}{r}\right),
 \end{equation}
  where $\lambda_f(n)$ denotes the $nth$ Hecke eigenvalue of $f$ (see equation (2) of \cite{KZ}).
 In particular if $p$ is prime this becomes
\[
 c(|d|p^2)=c(|d|)\bigg(\lambda_f(p)-\frac{\chi_d(p) }{\sqrt{p}}\bigg).
\]
Since there exists $p$ which depends at most on $k$ such that $\lambda_f(p) < -2/\sqrt{p}$ it follows that if $c(|d|) \neq 0$ then $c(|d|)$ and $c(|d|p^2)$ have opposite signs. 
 By Ono and Skinner's result there are $\gg X/(p^2 \log (X/p^2)) \gg X/\log X$ fundamental discriminants $|d| \le X/p^2$ such that $c(|d|) \neq 0$. So by considering the signs of the Fourier coefficients $c(n)$ at $n=|d| \le X/p^2$ with $c(|d|) \neq 0$ and at $m=|d|p^2 \le X$ we arrive at the claimed result.  \end{proof}

\section{Upper bounds for mollified moments}

Let $f$ be a level 1, Hecke cusp form of weight $2k$.
The aim of this section is to compute an upper bound for mollified moments of $L(\tfrac12,f \otimes \chi_d)$, conditionally under GRH. This gives an upper bound for the  mean square of the mollified Fourier coefficients of $g \in S_{k+1/2}^+$.

The second moment $L(\tfrac12, f \otimes \chi_d)$ has been asymptotically computed assuming GRH by Soundararajan and Young \cite{SoundYoung}. However, a direct adaptation of their method cannot handle the introduction of a mollifier of length $\gg |d|^{\varepsilon}$. For this reason we use a different approach, which is based on the refinement of Soundararajan's \cite{SoundMoments} method for upper bounds on moments due to Harper \cite{Harper}.

The main result is
\begin{proposition} \label{prop:moments} Assume GRH. Let $l, \kappa >0$ such that $l \cdot \kappa \in \mathbb N$ and $ \kappa \cdot l \le C$. Also, let $M(\cdot;\tfrac{1}{\kappa})=M(\cdot;\tfrac{1}{\kappa},x,C)$ be as in \eqref{eq:mollifierdef}. Then 
\[
\sumfund_{|d| \le x} L(\tfrac12, f \otimes \chi_d)^l  M(d; \tfrac{1}{\kappa})^{l \cdot \kappa} \ll x,
\]
where the sum is over fundamental discriminants and the implied constant depends on $f,l,\kappa$.
\end{proposition}

Using the proposition above we can easily deduce Proposition \ref{prop:secondmoment}.

\begin{proof}[Proof of Proposition \ref{prop:secondmoment}]
Using \eqref{eq:waldspurger} it follows that
\[
\sum_{\substack{1 \le n \le x \\ 2n \text{ is } \square-\text{free}}} |c(8n)|^4 M((-1)^k 8n;\tfrac12)^4 \ll \sumfund_{|d| \le 8x} L(\tfrac12, f\otimes \chi_d)^2 M( d;\tfrac12)^4 \ll x
\]
where the sum is over all fundamental discriminants. \end{proof}

\subsection{Preliminary results.} 
In this section we introduce our mollifier for the Fourier coefficients of $g$ at fundamental discriminants. The shape of the mollifier is motivated by Harper's refinement \cite{Harper} of Soundararajan's \cite{SoundMoments} bounds for moments.

\vspace{.1 in}

\paragraph{\bf{Notation.}}
Let $\lambda(n)=(-1)^{\Omega(n)}$ denote the Liouville function (which should not be confused with $\lambda_f(n)$). Denote by $\nu(n)$ the multiplicative function with $\nu(p^a)=\frac{1}{a!}$ and write $\nu_{j}(n)=(\nu \ast \cdots \ast \nu)(n)$ for the $j$-fold convolution of $
\nu$. A useful observation is that
\begin{equation} \label{eq:nu}
\nu_j(p^a)=\frac{j^a}{a!},
\end{equation}
which may be proved by induction or otherwise. Also for an interval $I$ and $m \in \mathbb Z$ let
\[
P_{I}(m; a( \cdot) )=\sum_{\substack{p \in I }} \frac{a(p)}{\sqrt{p}} \left( \frac{m}{p}\right).
\]

\vspace{.1 in}

\vspace{.1 in}
    \paragraph{\textbf{An $L$-function inequality}.} We now prove an inequality in the spirit of Harper's work on sharp upper bounds for moments of $L$-functions, in the context of quadratic twists of $L$-functions attached to Hecke cusp forms. The upshot of the inequality is it essentially allows us to almost always bound the $L$-function by a very short Dirichlet polynomial. 

Let us now introduce the following notation, which will be used throughout this section. Let $l, \kappa>0$ be such that $l \cdot \kappa \in \mathbb N$ and suppose that $1 \le l \cdot \kappa \le C$.
Also, let $\eta_1, \eta_2>0$ be sufficiently small in terms of $C$.
For $j=0, \ldots, J$ let 
\begin{equation} \label{eq:elldef}
\theta_j= \eta_1 \frac{e^j}{(\log \log x)^{5}} \qquad  \mbox{ and } \qquad \ell_j=2 \lfloor  \theta_j^{-3/4}  \rfloor,
\end{equation} 
where $J$ is chosen so that $ \eta_2  \le \theta_J \le e \eta_2 $ (so $J \asymp \log \log \log x$).
For $j=1, \ldots, J$ let $I_j=(x^{\theta_{j-1}}, x^{\theta_j}]$ and set $I_0=(c, x^{\theta_0}]$, where $c \ge 2$ is fixed. Note that the choice of parameters $\theta_{j}, \ell_j$ depends on $C$ and is independent of $l, \kappa$ satisfying $1 \le l \cdot \kappa \le C$.


For $j=0, \ldots, J$ and $t>0$ let 
\begin{equation} \notag
w(t;j)=\frac{1}{t^{\theta_j \log x}} \left(1-\frac{\log t}{\theta_j \log x} \right)
\end{equation}
and define the completely multiplicative function $a(\cdot; j)$ by
\begin{equation} \label{eq:adef}
a(p;j)=\lambda_f(p) w(p;j).
\end{equation}
The smooth weights $w(\cdot; j)$ appear for technical reasons and their effect is mild.

  Let $\ell$ be an positive even integer. For $t \in \mathbb R$, let
  \[
  E_{\ell}(t)=\sum_{s \le \ell} \frac{t^s}{s!}.
  \]
  Note $E_{\ell}(t) \ge 1$ if $t \ge 0$ and $E_{\ell}(t)>0$ for any $t \in \mathbb R$ since $\ell$ is even; the latter inequality may be seen using the Taylor expansion for $e^t$.
Moreover, using the Taylor expansion it follows for $t \le \ell/e^2$  that 
\begin{equation} \label{eq:taylor}
e^t \le (1+e^{-\ell/2})E_{\ell}(t).
\end{equation}

For each $j=0,\ldots, J$ let
\begin{equation} \label{eq:ddef}
D_j(m;l)=\prod_{r=0}^j (1+e^{-\ell_j/2}) E_{\ell_r}(lP_{I_r}(m;a( \cdot;j)))
\end{equation}
and note $D_j(m;l)>0$, since each term in the product is $>0$.
A useful observation is that for a positive integer $s$
\begin{equation} \label{eq:PID}
\begin{split}
    P_I(m; a(\cdot))^s=&\sum_{p_1, \ldots, p_s \in I} \frac{a(p_1 \cdots p_s)}{\sqrt{p_1 \cdots p_s}} \left( \frac{m}{p_1 \cdots p_s}\right)\\
    =&\sum_{\substack{p|n \Rightarrow p \in I \\ \Omega(n)=s}} \frac{a(n)}{\sqrt{n}} \left( \frac{m}{n} \right) \sum_{p_1\cdots p_s=n} 1 \\
    =& s! \sum_{\substack{p|n \Rightarrow p \in I \\ \Omega(n)=s}} \frac{a(n)}{\sqrt{n}} \left( \frac{m}{n} \right) \nu(n).
    \end{split}
\end{equation}
Additionally,
for any real number $l \neq 0$
\begin{equation} \label{eq:id}
\begin{split}
E_{\ell}(l P_{I}(m))=&\sum_{s \le \ell} \frac{l^s}{s!}  P_{I}(m)^s\\
=&\sum_{\substack{p|n \Rightarrow p \in I \\ \Omega(n) \le \ell}} \frac{ l^{\Omega(n)}  a(n) \nu(n)}{\sqrt{n}}  \left( \frac{m}{n} \right).
\end{split}
\end{equation}

We are now ready to state our main inequality for $L(\tfrac12,f\otimes \chi_d)$.
\begin{proposition} \label{prop:harper}
Assume GRH. Suppose that $ 1 \le |d| \le x$ is a fundamental discriminant and $l>0$ is a real number.
Also let
\[
A(d)=A(d;f)=\prod_{\substack{p|d \\ p>3}}\left(1+ \frac{\lambda_f(p^2)-1}{p}\right).
\]
Let $\ell_j,\theta_j, I_j$ and $a(\cdot,j)$ be as in \eqref{eq:elldef} and \eqref{eq:adef} (resp.).
Then there exists a sufficiently large absolute constant $C_1=C_1(l)>0$ such that for each $|d| \le x$ either:
\begin{equation}
|P_{I_{0}}(d;a(\cdot, j))| \ge \frac{ \ell_0}{l e^2}
\end{equation}
for some $0 \le j \le J$
or
\begin{equation} \label{eq:harperineq}
\begin{split}
&\left(A(d) \log x \right)^{l/2}
L(\tfrac12, f\otimes \chi_d)^l \\
&  \ll     D_J(d;l)+ 
\sum_{\substack{0 \le j \le J-1 \\ j+1 \le u \le J}}  \left(\frac{1}{\theta_{j}}\right)^{C_1} \exp\Big({\frac{3l}{\theta_{j}}} \Big) D_{j}(d;l) \left( \frac{ e^2 l P_{I_{j+1}}(d; a( \cdot , u))}{ \ell_{j+1}}\right)^{ s_{j+1}} ,
\end{split}
\end{equation}
for any even integers $s_1,\ldots,s_{J}$,
where the implied constant depends on $f$ and $l$.

\end{proposition}

\begin{proof}
Applying Theorem 2.1 of Chandee \cite{Phi} gives, for any $X \ge 10$, that
\begin{equation} \label{eq:soundbound}
L(\tfrac12, f \otimes \chi_d) \le \exp\left( \sum_{p^n \le X} \frac{\chi_d(p) (\alpha_p^n +\beta_p^n)}{n p^{n(\frac12 +\frac{1}{\log X})}} \frac{\log X/p^n}{\log X}+3 \cdot \frac{\log |d|}{\log X}+O(1) \right),
\end{equation}
where $\alpha_p,\beta_p$ are the Satake parameters (note $\alpha_p^n+\beta_p^n$ is real valued see Remark 2 of \cite{LR}). In the sum over primes, the contribution from the prime powers with $n \ge 3$ is bounded so that it may be included in the $O(1)$ term. Also, noting that $\alpha_p^2+\beta_p^2=\lambda_f(p^2)-1$ the squares of primes contribute
\begin{equation} \label{eq:squares}
\frac{1}{2} \cdot \sum_{\substack{p \le X \\ p \nmid d}} \frac{\lambda_f(p^2)-1}{p^{1+\frac{2}{\log X}}} \frac{\log X/p^2}{\log X}=-\frac{1}{2}\log \log X-\sum_{\substack{p \le X \\ p|d}} \frac{(\lambda_f(p^2)-1)}{2p}+O\left(1\right)
\end{equation}
where the implied constant depends on $f$.

For $r=0, \ldots, J$, let $\mathcal E_r$ be the set of fundamental discriminants $d$ such that $\max_{r \le  u \le J} |P_{I_r}(d;a(\cdot; u))| < \frac{ \ell_r}{ l e^2}$. 
For each $d$ we must have one of the following: (i) $d \notin \mathcal E_0$; (ii) $d \in \mathcal E_r$ for each $0 \le r \le J$; (iii) there exists $0 \le j \le J-1$ such that $d \in \mathcal E_r$ for $0 \le r \le j$ and $d \notin \mathcal E_{j+1}$. 
Hence, we get for each fundamental discriminant with $|d| \le x$
either:
\begin{eqnarray}
\max_{0 \le  u \le J} |P_{I_0}(d;a(\cdot; u))| &\ge& \frac{ \ell_0}{ l e^2}; \label{eq:conditionone} \\ 
\max_{j \le u \le J} |P_{I_j}(d;a(\cdot; u))| &<& \frac{ \ell_j}{ l e^2} \qquad \text{ for each } \qquad j =0, \ldots, J; \label{eq:conditiontwo}
\end{eqnarray}
or
for some $j=0, \ldots, J-1$ 
\begin{equation} \label{eq:conditionthree}
\max_{r \le u \le J} |P_{I_r}(d;a(\cdot; u))|  <  \frac{ \ell_r}{l e^2} \qquad \text{for each} \qquad r \le j
\end{equation}
and
\begin{equation} \label{eq:tbd}
\max_{j+1 \le u \le J} |P_{I_{j+1}}(d;a(\cdot; u))| \ge  \frac{ \ell_{j+1}}{ l e^2}.
\end{equation}
Here \eqref{eq:conditionone}, \eqref{eq:conditiontwo} correspond to possibilities (i) and (ii) above (resp.), while \eqref{eq:conditionthree} and \eqref{eq:tbd} corresponds to (iii).

If \eqref{eq:conditionone} holds for $d$ then we are done. If \eqref{eq:conditiontwo} holds, we apply \eqref{eq:soundbound} with $X=x^{\theta_J}$ and so the term $3 \cdot \frac{\log |d|}{\log X}$ is $\ll 1$. Also, the contribution from \eqref{eq:squares} is
\[
-\tfrac12\log \log x-\sum_{\substack{p \le X\\p|d}} \frac{(\lambda_f(p^2)-1)}{2p}+O(1).
\]
Hence, applying these estimates along with \eqref{eq:taylor} we get that
\begin{equation} \label{eq:case1bd}
\begin{split}
\left(A(d) \log x \right)^{l/2}
L(\tfrac12, f\otimes \chi_d)^l \ll& \exp\left( l \sum_{p \le x^{\theta_J}} \frac{a(p;J)}{\sqrt{p}} \chi_d(p)\right)\\
\ll &\prod_{j=0}^{J} \exp\left(l P_{I_j}(d;a(\cdot;J)) \right) \\
\le& \prod_{j=0}^J (1+e^{-\ell_j/2}) E_{\ell_j}(lP_{I_j}(m;a( \cdot;J)))
= D_J(d;l).\\
\end{split}
\end{equation}

Finally, if \eqref{eq:conditionthree} and \eqref{eq:tbd} hold
we apply \eqref{eq:soundbound} with $X=x^{\theta_j}$ 
and for each $0 \le r \le j$ we argue as above and bound the term $\exp(lP_{I_r}(d; a(\cdot; j)))$ by $(1+e^{-\ell_r/2})E_{\ell_j}(lP_{I_r}(d;a(\cdot, j)))$. 
In \eqref{eq:soundbound} we use the inequality $3 \cdot \frac{\log |d|}{\log X}\le \frac{3}{\theta_j}$ and note that the contribution from the squares of primes \eqref{eq:squares} is 
$$-\tfrac{1}{2}\log \log x-\sum_{\substack{ p|d}} \frac{(\lambda_f(p^2)-1)}{2p}+O\left(\log \frac{1}{\theta_j}\right).$$
So that applying these estimates, we get that for any non-negative, even integer $s_{j+1}$
\begin{equation} \label{eq:case2bd}
\begin{split}
& \left(A(d) \log x \right)^{l/2}
L(\tfrac12, f\otimes \chi_d)^l \\
& \qquad \ll \left(\frac{1}{\theta_{j}}\right)^{C_1} \exp(3l/\theta_{j})  D_{j}(d;l) \max_{j+1 \le u \le J}\left( \frac{ e^2 l P_{I_{j+1}}(d; a( \cdot , u))}{ \ell_{j+1}}\right)^{ s_{j+1}} ,
\end{split}
\end{equation}
where we have included the extra factor 
\[
\max_{j+1 \le u \le J}\left( \frac{ e^2 l P_{I_{j+1}}(d; a( \cdot , u))}{ \ell_{j+1}}\right)^{ s_{j+1}} \ge 1.
\]

Combining \eqref{eq:case1bd} and \eqref{eq:case2bd} and using the inequality $\max(a,b) \le a+b$ for $a,b>0$ gives \eqref{eq:harperineq}.
\end{proof}

\paragraph{\textbf{The definition of the mollifier.}}

With Proposition \ref{prop:harper} in mind we now choose our mollifier so that it will counteract the large values of $D_J(d;l)$. For $j=0, \ldots, J$ let
\[
M_j(m;\tfrac{1}{\kappa})=\sum_{\substack{p|n \Rightarrow p \in I_j \\ \Omega(n) \le \ell_j}} \frac{\kappa^{-\Omega(n)} a(n;J) \lambda(n)}{\sqrt{n}} \nu(n) \left( \frac{m}{n}\right),
\]
where $\ell_j,\theta_j$ and $I_j$ are as in \eqref{eq:elldef}.
Let
\begin{equation} \label{eq:mollifierdef}
M(m; \tfrac{1}{\kappa})
= (\log x)^{1/(2\kappa)}\prod_{0\le j \le J} M_j(m;\tfrac{1}{\kappa});
\end{equation}
and note that in Propositions \ref{prop:secondmoment} and \ref{prop:firstmoment} we specialize $I_0=(c,x^{\theta_0}]$ so that $c$ is sufficiently large, yet fixed.
A useful observation is that $M(m;\tfrac{1}{\kappa})>0$, which can be seen by using \eqref{eq:id} along with the fact that $E_{\ell}(t)>0$ for even $\ell$ and $t\in \mathbb R$.
Also, let $\delta_0=\sum_{j \le J } \ell_j \theta_j$.  Observe that by construction
$\delta_0>0$ is fixed, sufficiently small and the length of the Dirichlet polynomial $M(m;\frac{1}{\kappa})$ is $x^{\delta_0}$. Also,
\begin{equation} \label{eq:mollifieropen}
    M(m;\tfrac{1}{\kappa})^{l \cdot \kappa}=(\log x)^{l/2} \sum_{n \le x^{l \cdot \kappa \delta_0}} \frac{h(n) a(n;J) \lambda(n)}{\kappa^{\Omega(n) } \sqrt{n}} \left( \frac{m}{n}\right) 
\end{equation}
where
\begin{equation} \label{eq:hdef}
h(n)=\sum_{\substack{n_0 \cdots n_J=n \\ p|n_0 \Rightarrow p \in I_0, \ldots, p|n_J \Rightarrow p \in I_J \\ \Omega(n_0) \le l \cdot \kappa \ell_0,\ldots, \Omega(n_J) \le l \cdot \kappa \ell_J}} \nu_{l \cdot \kappa}(n_0; \ell_0) \cdots  \nu_{l \cdot \kappa}(n_J; \ell_J).
\end{equation}
and for $n,r, \ell \in \mathbb N$
\begin{equation} \label{eq:nualt}
\nu_{r}(n;\ell)= \sum_{\substack{n_1 \cdots n_r= n \\ \Omega(n_1) \le \ell, \ldots, \Omega(n_r) \le \ell}} \nu(n_1) \cdots \nu(n_r).
\end{equation}
Note that $\nu_r(n;\ell) \le \nu_r(n)$ and if $\Omega(n) \le \ell$ then $\nu_r(n;\ell)=\nu_r(n)$.
Also,
\begin{equation} \label{eq:mollifierbd}
    M(m;\tfrac{1}{\kappa})^{l \cdot \kappa} \ll x^{(l \cdot \kappa) \delta_0} (\log x)^{l/2}. 
\end{equation}

\subsection{The Proof of Proposition \ref{prop:moments}}

We first require the following fairly standard lemma, which follows from Poisson summation.
\begin{lemma} \label{lem:avg}
Let $F$ be a Schwartz function such that its Fourier transform $\widehat F$ has compact support in $(-A,A)$, for some fixed $A>0$. Also, let $I_0,\ldots, I_J \subset [1,x]$ be disjoint intervals of the form $I_0=(c, x^{\theta_0}]$, $I_j=(x^{\theta_{j-1}}, x^{\theta_j}]$ where $\theta_0,\ldots, \theta_J \in \mathbb R$ and $c \ge 1$ is fixed. Suppose $\ell_0, \ldots, \ell_J\ge 1$ are real numbers such that $\sum_{j=0}^J \ell_j \theta_j <1/2$. 
 Then for any arithmetic function $g$ we have
\begin{equation} \label{eq:easysum}
\sum_{m \in \mathbb Z} \prod_{0 \le j \leq J}  \sum_{ \substack{ p| n \Rightarrow p \in I_j \\ \Omega(n) \le \ell_j}  } \frac{g(n)  \left( \frac{m}{n}\right)}{\sqrt{n}}  F\left( \frac{m}{x}\right)
=\widehat F(0) x
\prod_{0 \le j \leq J} \sum_{\substack{ p| n \Rightarrow p \in I_j \\ \Omega(n^2) \le \ell_j} } \frac{g(n^2)}{n} \frac{\varphi(n^2)}{n^2}.
\end{equation}

\end{lemma}
\begin{proof}
The LHS of \eqref{eq:easysum} equals
\begin{equation} \label{eq:calc}
\sum_{ \substack{n_0, \ldots, n_J \in \mathbb N\\ p | n_0 \Rightarrow p \in I_0, \ldots, p|n_J \Rightarrow p \in I_j\\ \Omega(n_0) \le \ell_0, \ldots, \Omega(n_J) \le \ell_J}} \frac{g(n_0) \cdots g(n_J)}{\sqrt{n_0 \cdots n_J}} \sum_m \left( \frac{m}{n_0 \cdots n_J} \right) F\left(\frac{m}{x} \right) .
\end{equation}
Using Poisson summation, see Remark 1 of \cite{LR}, we get that for $n \le X^{1/2}$ if $n=\square$ that
\[
\sum_{m} \left( \frac{m}{n} \right) F\left(\frac{m}{x} \right)= \widehat F(0) x \frac{\varphi(n)}{n}
\]
and for $n \neq \square$ the above sum equals zero.
By assumption $n_0 \cdots n_J \le x^{\ell_0 \theta_0} \cdots x^{\ell_J \theta_J} < x^{1/2}$. Hence,
the inner sum in \eqref{eq:calc} equals zero unless $n_0 \cdots n_J =\square$ and since $(n_j,n_i)=1$ for $i \neq j$ in this case $n_0, \ldots, n_J =\square$. We conclude that the inner sum in \eqref{eq:calc} equals
\[
\widehat F(0) x \frac{\varphi(n_0 \cdots n_J)}{n_0 \cdots n_J}=\widehat F(0) x \frac{\varphi(n_0) \cdots \varphi(n_J)}{n_0 \cdots n_J}
\]
when $n_0, \ldots, n_J=\square$ and is zero otherwise, thereby giving the claim.
\end{proof}

To prove Proposition \ref{prop:moments} we will use Proposition \ref{prop:harper} and then apply Lemma \ref{lem:avg}. This leaves us with the problem of bounding the resulting sum. This task will be accomplished in the next two lemmas. 

Before stating the next lemma let us introduce the following notation. Given a statement $S$ let $1_S$ equal one if $S$ is true and zero otherwise.

\begin{lemma} \label{lem:diag} Let 
$a(\cdot),b(\cdot)$
be real valued, completely multiplicative functions with $a(p),b(p) \ll 1$.
 Then for each $j=0, \ldots, J$
\begin{equation} \label{eq:diagbd}
\begin{split}
& \bigg|\sum_{\substack{mn=\square \\ p|mn \Rightarrow p \in I_j \\ \Omega(m) \le  (l \cdot \kappa) \ell_j, \Omega(n) \le \ell_j }} \frac{l^{\Omega(n)} a(m)b(n) \lambda(m)}{\kappa^{\Omega(m)} \sqrt{mn}} \nu_{l \cdot \kappa}(m; \ell_j) \nu(n)  \cdot \frac{\varphi(mn)}{mn} \bigg| \\
& \qquad \qquad \le \left(1+O\left(\frac{1_{j=0} \cdot (\log x)^{O(1)}+1}{2^{\ell_j}} \right) \right) \prod_{\substack{p \in I_j }} \left(1-l^2\frac{\alpha(p)}{p}+O\left( \frac{1}{p^2}\right) \right),
\end{split}
\end{equation}
where
\begin{equation} \label{eq:alphadef}
\alpha(p)=\alpha(p;a(\cdot),b(\cdot))= a(p)b(p)-\tfrac12 a(p)^2-\tfrac12 b(p)^2.
\end{equation}
\end{lemma}

\begin{remark} Later we will take $a(\cdot)=a(\cdot;J)$ and $b=a(\cdot;j)$ with $j < J$ (see \eqref{eq:adef}). Observe that 
\begin{equation} \label{eq:easybd}
\begin{split}
\sum_{p \le x^{\theta_j}} \frac{\alpha(p)}{p}=&\sum_{p \le x^{\theta_j}} \frac{\lambda_f(p)^2}{p}\left(w(p;x^{\theta_j})w(p;x^{\theta_J})-\tfrac12w(p;x^{\theta_J})^2-\tfrac12w(p;x^{\theta_j})^2 \right) \\
\ll &
\sum_{p \le x^{\theta_j}} \frac{1}{p} \cdot \frac{\log p}{\theta_j \log x}\ll 1.
\end{split}
\end{equation}
\end{remark}

\begin{proof}
In the sum over $m,n$ on the LHS of \eqref{eq:diagbd}  write $m=g\cdot m_1$, $n=g \cdot n_1$ where $g=(m,n)$. Since $mn=\square$ it follows that $m_1=\square$ and $n_1=\square$. We will proceed by estimating the sum in terms of an Euler product. To this end, observe that if $\max\{\Omega(n),\Omega(m)\} \ge  \ell_j$ then $2^{\Omega(m)+\Omega(n)}/2^{\ell_j} \ge 1$; so using this along with the remarks following \eqref{eq:nualt} the sum on the LHS of \eqref{eq:diagbd} equals
\begin{equation} \label{eq:extend}
\begin{split}
&\sum_{\substack{p|mng \Rightarrow p \in I_j \\ (m,n)=1 
}} \frac{l^{\Omega(gn^2)} a(g)b(g)a(m)^2b(n)^2 \lambda(g)}{\kappa^{\Omega(gm^2)} gmn} \nu_{l \cdot \kappa}(gm^2) \nu(gn^2)  \cdot \frac{\varphi(g^2m^2 n^2)}{g^2m^2n^2}\\
&+O\left(\frac{1}{2^{\ell_j}} \sum_{\substack{ p|mng \Rightarrow p \in I_j }} \frac{2^{\Omega(gn^2)+\Omega(gm^2)} (l\cdot \kappa)^{\Omega(gn^2)} |a(g)b(g)|a(m)^2b(n)^2 }{\kappa^{\Omega(gm^2)}g mn}  \right),
\end{split}
\end{equation}
where we have trivially estimated $\nu, \nu_{l \cdot \kappa}, \varphi$, and $ \lambda$ in the second sum.
Recall for any multiplicative function $r(n)$, and $m \in \mathbb N$ whose prime factors all lie in an interval $I$ (which also contains at least one prime) that
\[
\sum_{p|n \Rightarrow p \in I} \frac{r(mn)}{n}= \prod_{\substack{p \in I \\ p^{a} || m}} \sum_{t \ge 0} \frac{r(p^{t+a})}{p^t}.
\]
It follows that the error term in \eqref{eq:extend} is
\begin{equation} \label{eq:errorbd1}
\ll \frac{1}{2^{\ell_j}} \prod_{p \in I_j} \left(1+O\left( \frac1p \right)\right) \asymp \begin{cases}
\frac{1}{2^{\ell_j}} \left( \frac{\theta_{j}}{\theta_{j-1}}\right)^{O(1)} \asymp \frac{1}{2^{\ell_j}} & \text{ if } j \neq 0,\\
\frac{(\log x)^{O(1)}}{2^{\ell_0}} & \text{ if } j = 0.
\end{cases}
\end{equation}

To estimate the main term in \eqref{eq:extend} consider
\[
\sigma(p;u,v)=\sum_{t \ge 0} \left(\frac{l}{\kappa}\right)^{t} \frac{a(p)^{t}b(p)^t (-1)^t }{p^t} \nu_{l \cdot \kappa}(p^{t+u}) \nu(p^{t+v}) \frac{\varphi(p^{2t+u+v})}{p^{2t+u+v}}
\]
and
\[
f_1(n)=\prod_{p^u || m} \frac{\sigma(p;u,0)}{\sigma(p;0,0)}, \qquad f_2(n)= \prod_{p^v || n}\frac{\sigma(p;0,v)}{\sigma(p;0,0)}.
\]
Note that $\sigma(p;u,v)>0$ for $p$ sufficiently large.
Hence, evaluating the sum over $g$ in the main term in \eqref{eq:extend} we see that it equals
\begin{equation} \label{eq:evaluate}
\begin{split}
&\bigg(\prod_{\substack{p \in I_j } } \sigma(p;0,0) \bigg) \sum_{\substack{p|mn \Rightarrow p \in I_j \\ (m,n)=1}} \frac{l^{2 \Omega(n)}a(m)^2b(n)^2 }{\kappa^{2\Omega(m) } mn} f_1(m^2) f_2(n^2) \\
& \le \bigg(\prod_{\substack{p \in I_j } } \sigma(p;0,0) \bigg) \sum_{\substack{p|mn \Rightarrow p \in I_j }} \frac{l^{2 \Omega(n)}a(m)^2b(n)^2 }{\kappa^{2\Omega(m) } mn} f_1(m^2) f_2(n^2) \\
&= \prod_{p \in I_j} \bigg( \sigma(p;0,0) \bigg(\sum_{t \ge 0} \frac{ a(p)^{2t} f_1(p^{2t})}{\kappa^{2t}p^{t}}\bigg) \bigg( \sum_{t \ge 0} \frac{l^{2t}b(p)^{2t} f_2(p^{2t})}{ p^t} \bigg)   \bigg) .
\end{split}
\end{equation}
Note  $f_{1}(p^u), f_2(p^u) \ll_{l,\kappa} 1$ and recall \eqref{eq:nu}. It follows that
\[
\begin{split}
\sigma(p;0,0)=&1-\left(\frac{l}{\kappa} \cdot l \kappa\right) \frac{a(p)b(p)}{p}+O\left( \frac{1}{p^2}\right), \\
f_1(p^2)=& \frac{(l\cdot \kappa)^2}{2}+O\left( \frac{1}{p}\right), \\
f_2(p^2)=& \frac{1}{2}+O\left(\frac{1}{p} \right).
\end{split}
\]
Using these estimates on the RHS of \eqref{eq:evaluate} we see that the main term in \eqref{eq:extend} is
\begin{equation} \label{eq:evaluate1}
\le \prod_{p \in I_j} \left(1-l^2 \cdot \frac{a(p)b(p)}{p}+\frac{l^2}{2} \cdot \frac{a(p)}{p}+\frac{l^2}{2}\cdot \frac{b(p)^2}{p}+O\left( \frac{1}{p^2}\right) \right) 
\end{equation}
as claimed.
\end{proof}

\begin{lemma} \label{lem:mild} Let  $a(\cdot), b(\cdot)$
be real valued, completely multiplicative functions with $a(p), b(p) \ll 1$. 
 Then for each $j=0, \ldots J$
 and any even integer $s \ge 4(l \cdot \kappa ) \ell_j$
\begin{equation} \label{eq:mild}
\begin{split}
& \bigg| \sum_{\substack{mn=\square \\ p|mn \Rightarrow p \in I_{j} \\ \Omega(m) \le (l \cdot \kappa )\ell_j, \Omega(n)=s }} \frac{a(m)b(n) \lambda(m)}{\kappa^{\Omega(m)} \sqrt{mn}} \nu_{\kappa \cdot l}(m; \ell_j) \nu(n) \frac{\varphi(mn)}{mn} \bigg| \\
& \qquad \qquad \qquad \ll \frac{1_{j=0} \cdot (\log x)^{O(1)}+1}{2^{s/2} \lfloor \frac{3}{8}s \rfloor !} \left(1+\sum_{p \in I_j} \frac{b(p)^2}{p} \right)^{s/2}.
\end{split}
\end{equation} 
\end{lemma}
\begin{proof}
In the sum on the LHS of \eqref{eq:mild} write  $m=gm_1$, $n=gn_1$ where $g=(m,n)$, so that $mn =\square$ implies that $m_1=\square$ and $n_1=\square$. Also recall that $\nu_{\kappa \cdot l}(m;\ell_j) \le \nu_{\kappa\cdot l}(m)$. Hence, this sum  is
\begin{equation} \label{eq:crudebd}
\ll \sum_{\substack{p|gmn \Rightarrow p \in I_j \\ \Omega(g) \le (l \cdot \kappa )\ell_j, \Omega(gn^2)=s}} \frac{|a(g)b(g)| a(m)^2 b(n)^2}{ \kappa^{2\Omega(m)+\Omega(g)}gmn} \cdot  \nu_{l \cdot \kappa}(g) \nu_{l \cdot \kappa}(m^2) \nu(n^2),
\end{equation}
where we have also used the estimates $\nu_r(ab) \le \nu_r(a)\nu_r(b)$ and $\nu(a) \le 1$, for $a,b, r \in \mathbb N$, which follow from \eqref{eq:nu}.
The sum over $m$ is 
\begin{equation} \label{eq:msum}
 \ll  \sum_{p|m \Rightarrow p \in I_j} \frac{a(m)^2 \nu_{l \cdot \kappa}(m^2)}{\kappa^{2\Omega(m) }m} \ll \prod_{p \in I_j} \left(1+ \frac{(l \cdot \kappa)^2}{2} \cdot \frac{a(p)^2}{\kappa^2 p} \right) \ll
 \begin{cases}
\left( \frac{\theta_{j}}{\theta_{j-1}}\right)^{l^2/2} \asymp 1 & \text{ if } j \neq 0,\\
(\log x)^{O(1)} & \text{ if } j = 0.
\end{cases}
\end{equation}
Hence, using \eqref{eq:msum} along with the inequality $\nu(n^2) \le \nu(n) 2^{-\Omega(n)}$ it follows that \eqref{eq:crudebd} is
\begin{equation} \label{eq:gnsum}
\begin{split}
&\ll (1_{j=0} (\log x)^{O(1)} +1) \sum_{\substack{p|g \Rightarrow p \in I_j \\ \Omega(g) \le (l \cdot \kappa )\ell_j  }} \frac{|a(g)b(g)|}{\kappa^{\Omega(g)} g} \nu_{l \cdot \kappa}(g) \sum_{\substack{p|n \Rightarrow p \in I_j \\ \Omega(n)=(s-\Omega(g))/2}} \frac{b(n)^2}{2^{\Omega(n)}n} \nu(n)\\
&= (1_{j=0} (\log x)^{O(1)} +1) \sum_{\substack{p|g \Rightarrow p \in I_j \\ \Omega(g) \le (l \cdot \kappa ) \ell_j \\ 2| \Omega(g) }} \frac{|a(g)b(g)|}{\kappa^{\Omega(g)} g}  \nu_{l \cdot \kappa}(g) \left[ \frac{1}{((s-\Omega(g))/2)!}\left( \frac{1}{2} \sum_{p \in I_j} \frac{b(p)^2}{p}\right)^{(s-\Omega(g))/2}\right],
\end{split}
\end{equation}
by \eqref{eq:PID}. 
Using the assumption that $s \ge 4(l \cdot \kappa ) \ell_j$ gives $s-\Omega(g) \ge \frac{3}{4}s$ for $g$ with $\Omega(g) \le (l \cdot \kappa ) \ell_j$. This allows us to bound the bracketed term on the RHS of \eqref{eq:gnsum} by 
\[
2^{\Omega(g)/2} \cdot \frac{1}{2^{s/2} \lfloor \frac{3}{8}s \rfloor !} \left(1+\sum_{p \in I_j} \frac{b(p)^2}{p} \right)^{s/2}.
\]
Applying this estimate in \eqref{eq:gnsum}
it follows that there exists $B>0$, which depends at most on $l,\kappa$, such that the RHS of \eqref{eq:gnsum} is bounded by
\[
\begin{split}
&\ll \frac{(1_{j=0} (\log x)^{O(1)} +1)}{2^{s/2} \lfloor \frac{3}{8}s \rfloor !} \left(1+\sum_{p \in I_j} \frac{b(p)^2}{p} \right)^{s/2} \sum_{\substack{p|g \Rightarrow p \in I_j}} \frac{B^{\Omega(g)}}{g}\\ &\ll \frac{(1_{j=0} (\log x)^{O(1)} +1)}{2^{s/2} \lfloor \frac{3}{8}s \rfloor !} \left(1+\sum_{p \in I_j} \frac{b(p)^2}{p} \right)^{s/2}.
\end{split}
\]
\end{proof}

In the next lemma we estimate averages of our mollifier $M$ as defined in \eqref{eq:mollifierdef} against
the terms which appear in Proposition \ref{prop:harper}.

\begin{lemma} \label{lem:avgbds}
For $j=0,\ldots, J$ let $s_{j}$ be an even integer with $4(l\cdot \kappa) \ell_j \le s_j \le \frac{2}{5\theta_j}$. Then we have the following estimates:
\begin{align} \label{eq:likelybd}
& \sumfund_{|d| \le x} D_J(d;l)  M(d; \tfrac{1}{\kappa})^{l \cdot \kappa} \ll x (\log x)^{l/2}; \\
\label{eq:largedevbds}
& \sumfund_{|d| \le X}  M(d; \tfrac{1}{\kappa})^{l \cdot \kappa} P_{I_0}(d;a(\cdot;u))^{s_0} \ll
x (\log x)^{O(1)} \frac{s_0!}{2^{s_0/2} \lfloor \frac38 s_0 \rfloor !} \left(\sum_{p \in I_0} \frac{a(p;u)^2}{p} \right)^{s_0/2}; \\
& \label{eq:unlikelybd} \sumfund_{|d| \le X} D_j(d;l) P_{I_{j+1}}(d;a(\cdot; u))^{s_{j+1}} M(d; \tfrac{1}{\kappa})^{l \cdot \kappa}  \\
& \qquad \ll  x (\log x)^{l/2} e^{2l^2(J-j)} \frac{s_{j+1}!}{2^{s_{j+1}/2}\lfloor \frac{3}{8} s_{j+1} \rfloor !} \left(\sum_{p \in I_{j+1}} \frac{a(p;u)^2}{p}\right)^{s_{j+1}/2} \notag
\end{align}
for each $j=0, \ldots, J-1$.
The implied constants depend at most on $f,\kappa,l$ (and not on $j,u$).
\end{lemma}
\begin{proof}
We will first prove \eqref{eq:unlikelybd}, which is the most complicated of the bounds, and at the end of the proof we will indicate how to modify  the argument to establish \eqref{eq:likelybd} and \eqref{eq:largedevbds}. Recall that $D_j$ and $M$ are positive (see the remarks after \eqref{eq:ddef} and \eqref{eq:mollifierdef}). 
So for any Schwartz function $F$ which majorizes $1_{[-1,1]}$, such that $\widehat F$ has compact support in $(-A,A)$ the LHS of \eqref{eq:unlikelybd} is
\[
\begin{split}
 \le \sum_{m \in \mathbb Z}   D_j(m;l) P_{I_{j+1}}(m;a(\cdot; u))^{s_{j+1}} M(m; \tfrac{1}{\kappa})^{l \cdot \kappa}  F\left( \frac{m}{x}\right).
\end{split}
\]
The Dirichlet Polynomial above is supported on integers $n \le x^{\theta}$ where $\theta= l \cdot \kappa \sum_{0 \le r \le J} \theta_r \ell_r+ \sum_{0 \le r \le j} \theta_r\ell_r+s_{j+1} \theta_{j+1} <1/2$. Hence, using \eqref{eq:ddef}, \eqref{eq:PID}, \eqref{eq:id}, \eqref{eq:mollifieropen}, \eqref{eq:hdef} and applying Lemma \ref{lem:avg} the above equation is bounded by
\begin{equation} \label{eq:unlikely}
\begin{split}
 \le x (\log x)^{l/2} \widehat F(0) \prod_{0 \le r \le J} \sum_{\substack{mn=\square \\ p|mn \Rightarrow p \in I_r \\ \Omega(m) \le (l \cdot \kappa) \ell_r  }} \frac{a(m;J) g(n;u,r) \lambda(m)}{\kappa^{\Omega(m)} \sqrt{mn}} \nu_{\kappa \cdot l}(m;\ell_r) \nu(n)  \frac{\varphi(mn)}{mn}
\end{split}
\end{equation}
where
\[
g(n;u,r)=\begin{cases}
 1_{\Omega(n) \le \ell_r} \cdot  l^{\Omega(n)} \cdot a(n;j) & \text{ if } 0 \le r \le j, \\
  1_{\Omega(n)=s} \cdot s! \cdot a(n;u)   & \text{ if } r=j+1, \\
  1_{n=1} & \text{ if } j +1 < r \le J. 
\end{cases}
\]

Let $\alpha(\cdot)$ be as in \eqref{eq:alphadef} and take $a(\cdot)=a(\cdot;J)$, $b(\cdot)=a(\cdot;j)$. Applying Lemma \ref{lem:diag},
the contribution to \eqref{eq:unlikely} from the product over $0 \le r \le j$
is bounded in absolute value by
\begin{equation} \label{eq:rest1}
 \prod_{0 \le r \le j} \left(1+O\left( \frac{1_{r=0} \cdot (\log x)^{O(1)}+1}{2^{\ell_j/2}}\right)\right) \prod_{p \in I_j}\left(1-l^2\frac{\alpha(p)}{p}+O \left( \frac{1}{p^2}\right) \right)  \ll 1,
\end{equation} 
where in the last step we applied \eqref{eq:easybd}. Applying Lemma \ref{lem:mild} with $a(\cdot)=a(\cdot;J)$ and $b(\cdot)=a( \cdot;u)$  the term in \eqref{eq:unlikely} with 
$r=j+1$ contributes
\begin{equation} \label{eq:rest3}
\ll \frac{s_{j+1}!}{2^{s_{j+1}/2}\lfloor \frac{3}{8} s_{j+1} \rfloor !} \left(\sum_{p \in I_{j+1}} \frac{a(p;u)^2}{p}\right)^{s_{j+1}/2}.
\end{equation} 

It remains to handle the factors in the product in \eqref{eq:unlikely} with $j+1< r \le J$. For such factors, $m$ is a square so these terms are bounded by
\begin{equation} \label{eq:rest4}
\begin{split}
& \ll \prod_{j+1 < r \le J} \sum_{\substack{p|m \Rightarrow p \in I_r}} \frac{a(m;J)^2}{\kappa^{2\Omega(m)} m} \nu_{l \cdot \kappa}(m^2) \\
& \ll \prod_{x^{\theta_{j}} < p \le x^{\theta_J}} \left(1+ \frac{(l \cdot \kappa)^2}{2} \cdot \frac{\lambda_f(p)^2}{\kappa^2 p} \right) \ll \left( \frac{\theta_J}{\theta_j}\right)^{2l^2}=e^{2l^2(J-j)},
\end{split}
\end{equation} 
where we used \eqref{eq:nu} and the bound $|\lambda_f(p)| \le 2$, in the last steps.

Using \eqref{eq:rest1},  \eqref{eq:rest3} and \eqref{eq:rest4} in \eqref{eq:unlikely}
completes the proof of \eqref{eq:unlikelybd}. To establish \eqref{eq:likelybd} we repeat the same argument, the only differences are that in \eqref{eq:rest1} the product is over all $0 \le r \le J$ and the terms estimated in \eqref{eq:rest3} and \eqref{eq:rest4} do not appear. To establish \eqref{eq:largedevbds} note that the term estimated in \eqref{eq:rest1} does not appear and in the bound \eqref{eq:rest4} the only difference is that $j=0$, so we bound this term by $O((\log x)^{O(1)})$. Finally in place of \eqref{eq:rest3} we get by using Lemma \ref{lem:mild} the bound
\[
(\log x)^{O(1)}\frac{s_{0}!}{2^{s_{0}/2}\lfloor \frac{3}{8} s_{0} \rfloor !} \left(\sum_{p \in I_{0}} \frac{a(p;u)^2}{p}\right)^{s_{0}/2}.
\]
From these estimates \eqref{eq:largedevbds} follows.
\end{proof}

\begin{proof}[Proof of Proposition \ref{prop:moments}]
Let $A(d)$ be as in the statement of Proposition \ref{prop:harper}. In particular, using $|\lambda_f(p^2)|\le d(p^2)=3$ and $\lambda_f(p^2) =\lambda_f(p)^2 -1\ge-1$ it follows that
\[
 \left(\frac{\varphi(d)}{d}  \right)^2 \ll A(d) \ll \left( \frac{d}{\varphi(d)} \right)^2.
\]
Using this observation we note that
it suffices to prove that
\begin{equation} \label{eq:momentbd}
 \sumfund_{|d| \le x} \left( A(d)^{1/2}  L(\tfrac12, f \otimes \chi_d) \right)^l  M(d; \tfrac{1}{\kappa})^{l \cdot \kappa} \ll x,
\end{equation}
for $l>0$ (note the definition of $M$ is independent of $l$, for $1 \le l \cdot \kappa \le C$). Since by Cauchy-Schwarz \eqref{eq:momentbd} implies 
\[
\begin{split}
&\sumfund_{|d| \le x}  L(\tfrac12, f \otimes \chi_d)^l  M(d; \tfrac{1}{\kappa})^{l \cdot \kappa}\\
& \ll  \left(  \sumfund_{|d| \le x} \frac{1}{A(d)^{l}} \right)^{1/2} \left( \sumfund_{|d| \le x} \left( A(d)^{1/2} L(\tfrac12, f \otimes \chi_d) \right)^{2l}  M(d; \tfrac{1}{\kappa})^{2 l \cdot \kappa} \right)^{1/2} \ll x.
\end{split}
\]

We will now establish \eqref{eq:momentbd}. 
For $j=0,\ldots, J$, let $s_j=2 \lfloor \frac{1}{5 \theta_j}  \rfloor$.
We first divide the sum over $|d| \le x$ into two sums depending on whether 
\begin{equation} \label{eq:event1bd}
|P_{I_0}(d;a(\cdot;u))| \ge \ell_0/(le^2)
\end{equation}
for some $0 \le u \le J$.
To bound the contribution to the LHS of \eqref{eq:momentbd} of the terms with $|d| \le x$ for which \eqref{eq:event1bd} holds,
we use the bound $A(d) \ll (\log \log x)^2$, Chebyshev's inequality and then apply Cauchy-Schwarz to see that
\[
\begin{split}
& \sumfund_{\substack{|d| \le x  \\ |P_{I_0}(d;a(\cdot;u))| \ge \ell_0/(le^2)}} (A(d)^{1/2}L(\tfrac12, f \otimes \chi_d))^l M( d; \tfrac{1}{\kappa})^{l \cdot \kappa} \\
&  \ll (\log \log x)^{l} \sumfund_{\substack{|d| \le x }}  L(\tfrac12, f \otimes \chi_d)^l M( d; \tfrac{1}{\kappa})^{l \cdot \kappa} \left(  \frac{le^2 P_{I_0}(d;a(\cdot;u)}{\ell_0}\right)^{s_0} \\
&  \le  (\log \log x)^{l} \left(\sumfund_{\substack{|d| \le x }}  L(\tfrac12, f \otimes \chi_d)^{2l} \right)^{1/2} \left(  \frac{(l e^2)^{2s_0}}{\ell_0^{2s_0}} \sumfund_{\substack{|d| \le x }} M(d; \tfrac{1}{\kappa})^{2 l \cdot \kappa} P_{I_0}(d;a(\cdot;u))^{2s_0} \right)^{1/2} .
\end{split}
\]
On the RHS, the first sum is $\ll x (\log x)^{O(1)}$ by Soundararajan's \cite{SoundMoments} method for upper bounds for moments (see the example at the end of Section 4 of \cite{SoundMoments}). Using Lemma \ref{lem:avgbds}, \eqref{eq:largedevbds}, applying Stirling's formula and recalling the definition of $\ell_0,\theta_0$ (see \eqref{eq:elldef}) the second sum on the RHS is 
\[
\begin{split}
\ll &  x (\log x)^{O(1)} \left( \frac{100 l^2 s_0^{5/4}}{\ell_0^2} \right)^{s_0}\left(  \sum_{p \le x} \frac{\lambda_f(p)^2}{p}\right)^{s_0} \\
\ll & x (\log x)^{O(1)} \left( 100 l^2 \theta_0^{1/4} \log \log x \right)^{s_0} 
\ll  x (\log x)^{O(1)} \exp\left( -(\log \log x)^4 \right) \ll \frac{x}{(\log x)^A}
\end{split}
\]
for any $A \ge 1$. Hence, the contribution to \eqref{eq:momentbd} from $ |d| \le x$ satisfying \eqref{eq:event1bd} for some $u \le J$ is
$ \ll J \frac{x}{(\log x)^A} \ll \frac{x}{(\log x)^A} \log \log x
$
for any $A \ge 1$.

For the remaining fundamental discriminants $|d| \le x$ we apply \eqref{eq:harperineq} to see that their contribution to 
\eqref{eq:momentbd} is bounded by
\begin{equation} \label{eq:together}
\begin{split}
&\ll \frac{1}{(\log x)^{l/2}} \Bigg( \sumfund_{|d| \le x} D_J(d;l) M( d;\tfrac{1}{\kappa})^{l \cdot \kappa} \\
&+ \sum_{\substack{0 \le j \le J-1 \\ j+1 \le u \le J}}  \left(\frac{1}{\theta_{j}}\right)^{C_1} \exp\Big({\frac{3l}{\theta_{j}}} \Big)  \left( \frac{ e^2 l }{ \ell_{j+1}}\right)^{ s_{j+1}} 
\sumfund_{|d|\le x} P_{I_{j+1}}(d; a( \cdot , u))^{s_{j+1}} D_{j}(d;l) M( d; \tfrac{1}{\kappa})^{l \cdot \kappa} \Bigg).
\end{split}
\end{equation}

To complete the proof it suffices to show that the expression above is $\ll x$.
Applying \eqref{eq:likelybd} the first term in \eqref{eq:together} is
\begin{equation} \label{eq:togetherbd1}
\frac{1}{(\log x)^{l/2}} \sumfund_{|d| \le x} D_J(d;l) M(d; \tfrac{1}{\kappa})^{l \cdot \kappa} \ll x.
\end{equation}

Using \eqref{eq:unlikelybd} the second term in \eqref{eq:together} is bounded by
\begin{equation*}
\begin{split}
 \ll x \sum_{\substack{0 \le j \le J-1 \\ j+1 \le u \le J}} \left(\frac{1}{\theta_{j}}\right)^{C_1} \exp\Big({\frac{3l}{\theta_{j}}} \Big)  \left( \frac{ e^2 l }{ \ell_{j+1}}\right)^{ s_{j+1}}  \cdot e^{2l^2(J-j)} \frac{(s_{j+1})!}{2^{s_{j+1}/2}\lfloor \frac{3}{8} s_{j+1} \rfloor !} \left( \sum_{p \in I_{j+1}} \frac{a(p;u)^2}{p}\right)^{s_{j+1}/2} .
\end{split}
\end{equation*}
Applying Stirling's formula and estimating the inner sum over primes trivially as $\le 5 \log (\frac{\log x^{\theta_{j+1}}}{\log x^{\theta_j}})=5$, we see that the above is
\begin{equation} \label{eq:togetherbd3}
\ll  x \sum_{0 \le j \le J-1} (J-j)  \left(\frac{1}{\theta_{j}}\right)^{C_1} e^{2l^2(J-j)} \exp\Big({\frac{3l}{\theta_{j}}} \Big) \left( \frac{e^2l}{\ell_{j+1}} \cdot \left(\frac{8}{3} \right)^{3/8}  \cdot \frac{ s_{j+1}^{5/8}}{ \sqrt{2} e^{5/8}} \cdot 5\right)^{s_{j+1}}.
\end{equation}
By construction  
$
 \displaystyle \frac{s_{j+1}^{5/8}}{\ell_{j+1}}  \ll \theta_{j+1}^{1/8}.
$
Hence, there exists $c>0$ (which may depend on $l, \kappa$) such that \eqref{eq:togetherbd3} is
\begin{equation} \label{eq:lastbdinproof}
    \begin{split}
&\ll x \sum_{0 \le j \le J-1} (J-j) e^{2l^2(J-j)} \exp\left( -\frac{c}{\theta_j} \log \frac{1}{\theta_j}+O\left(\frac{1}{\theta_j} \right) \right) \\
&\ll x \sum_{1 \le j \le J} j e^{2l^2j} \exp\left(-\frac{c}{2} \cdot \frac{je^{j}}{\theta_J} \right) 
\ll x,
\end{split}
\end{equation}
where we used that $\theta_J/\theta_j=e^{J-j}$ in the last step.
Combining \eqref{eq:lastbdinproof} with \eqref{eq:togetherbd1}, gives that \eqref{eq:together} is $\ll x$, which completes the proof.
\end{proof}

\section{The Proof of Proposition \ref{prop:firstmoment}}

The main input into the proof of Proposition \ref{prop:firstmoment} is a twisted first moment of $L(\tfrac12,f \otimes \chi_d)$. Similar moment estimates were established in \cite{SoundNonvanishing}, \cite{SoundYoung} and \cite{RadziwillSound} and our proof closely follows the methods developed in those papers. We will include a proof of the following result for completeness.
\begin{proposition} \label{prop:twisted}
Let $\phi$ be a Schwartz function with compact support in the positive reals.
Also, let $u=u_1u_2^2$, be odd where $u_1$ is square-free. Then
\begin{equation} \label{eq:twist}
\begin{split}
\sum_{\substack{m \\ 2m \text{ is } \square-\text{free}}} L(\tfrac12, f \otimes \chi_{8m}) \chi_{8m}(u) \, \phi\left( \frac{(-1)^k 8m}{x}\right)=& C \int_{0}^{\infty} \phi(\xi) d\xi \cdot \frac{x}{u_1^{1/2}} \cdot \vartheta(u) \\
&+O\left( x^{7/8+\varepsilon} u^{3/8+\varepsilon} \right) 
\end{split}
\end{equation}
where $C>0$ depends only on $f$ and $\vartheta(\cdot)$ is a multiplicative function with $\vartheta(p^{2j+1})=\lambda_f(p)+O(1/p)$
and $\vartheta(p^{2j})=1+O(1/p)$.
\end{proposition}

The constant $C$ is explicitly given in the proof below, see \eqref{eq:Cdef}.
Before proving Proposition \ref{prop:twisted}, we will use the result to deduce Proposition \ref{prop:firstmoment}.
\subsection{The Proof of Proposition \ref{prop:firstmoment}}
We will only prove the lower bound
\[
\sum_{\substack{n \le x \\ 2n \text{ is } \square-\text{free}}} |c(8n)|^2 M((-1)^k 8n;\tfrac{1}{2})^2  \gg x
\]
since the proof of the upper bound is similar.  Using \eqref{eq:waldspurger} it follows that for a Schwartz function $\phi(\cdot)$ which minorizes $1_{[1/2,1]}(\cdot)$ with $\widehat \phi(0)>0$ the sum above is
\begin{equation} \label{eq:wlbd}
 \ge \frac{(k-1)!}{\pi^k}\cdot \frac{\langle g,g\rangle}{\langle f, f \rangle} \sum_{\substack{m \\ 2m \text{ is } \square-\text{free}}} L(\tfrac12, f\otimes\chi_{8m}) M( 8m;\tfrac12)^2 \phi\left(\frac{(-1)^k 8m}{x} \right).
\end{equation}
To proceed we now expand $M(8m;\tfrac12)^2$ and see that it equals
\begin{equation} \label{eq:mopen}
M(8m;\tfrac12)^2=(\log x)^{1/2} \sum_{n \le x^{2 \delta_0}} \frac{h(n) a(n;J) \lambda(n)}{2^{\Omega(n) } \sqrt{n}} \left( \frac{8m}{n}\right) 
\end{equation}
where $h(n)$ is as in \eqref{eq:hdef} (note that here $l \cdot \kappa =2)$.
Applying Proposition \ref{prop:twisted} gives that the RHS of \eqref{eq:wlbd} equals
\begin{equation} \label{eq:lbdform}
\begin{split}
  &  C^{'}  x  (\log x)^{1/2}  \sum_{\substack{n_1n_2^2 \le x^{2\delta_0} \\ n_1 \text{ is } \square-\text{free}}} \frac{h(n_1n_2^2) a(n_1n_2^2;J) \lambda(n_1) \vartheta(n_1n_2^2)}{2^{\Omega(n_1)+2\Omega(n_2)} n_2n_1} +O(x^{9/10})  \\
  & \qquad   = C^{'}  x  (\log x)^{1/2}  \prod_{0 \le j \le J} \sum_{\substack{p|mn^2 \Rightarrow p \in I_j \\ \Omega(mn^2) \le 2\ell_j}} \frac{a(m n^2;J) \lambda(m) \vartheta(mn^2) \mu(m)^2}{2^{\Omega(m)+2\Omega(n)} mn} \nu_2(mn^2; \ell_j)+O(x^{9/10}),
    \end{split}
\end{equation}
where $  C^{'}  =  C^{'}  (f,g,k, \phi)>0$ and in the last step we used that $\vartheta$ is a multiplicative function.

We will now estimate the inner sum on the RHS of \eqref{eq:lbdform}.
First note that $\vartheta(mn^2) \ll d(m) (\frac{mn^2}{\varphi(m)\varphi(n^2)})^{O(1)}$ and recall the remarks after \eqref{eq:nualt}, so we get that the sum equals
\begin{equation} \label{eq:extendsum}
\begin{split}
& \sum_{\substack{p|mn^2 \Rightarrow p \in I_j }} \frac{a(m n^2;J) \lambda(m) \vartheta(mn^2) \mu(m)^2}{2^{\Omega(m)+2\Omega(n)} mn} \nu_2(mn^2) \\
& \qquad \qquad
+O\left( \frac{1}{4^{\ell_j}} \sum_{p|nm \Rightarrow p \in I_j} \frac{a(m;J)a(n;J)^2 d(m)  \nu_2(m)\nu_2(n^2)}{mn} \cdot \left(\frac{mn^2}{\varphi(m)\varphi(n^2)}\right)^{O(1)} \right).
\end{split}
\end{equation}
The error term is
\begin{equation} \label{eq:errorbd}
\ll \frac{1}{4^{\ell_j}} \prod_{p \in I_j} \left( 1+O\left(\frac{1}{p} \right)\right) \ll \frac{1_{j=0} \cdot (\log x)^{O(1)} +1}{4^{\ell_j}}.
\end{equation}
To estimate the main term in \eqref{eq:extendsum}
consider
\[
s(p;a)= \sum_{t \ge 0} \frac{a(p;J)^{2t}}{4^t p^t} \vartheta(p^{2t+a})\nu_2(p^{2t+a}).
\]
Recall $\vartheta(p)=\lambda_f(p)+O(1/p)$ and $\vartheta(p^2)=1+O(1/p)$.
Evaluating the sum over $n$ in the main term of \eqref{eq:extendsum} we see that it equals
\begin{equation} \label{eq:eulereval}
\begin{split}
& \left(\prod_{p \in I_j} s(p;0)\right) \sum_{p|m \Rightarrow p \in I_j} \frac{a(m;J)\mu(m) }{2^{\Omega(m)} m} \prod_{p^a || m} \frac{s(p;a)}{s(p;0)}\\
& \qquad \qquad = \prod_{p \in I_j}
\left(1+\frac{a(p;J)^2}{2p}-\frac{\lambda_f(p)a(p;J)}{p}+O\left(\frac{1}{p^2} \right) \right).
\end{split}
\end{equation}
Using \eqref{eq:errorbd}, \eqref{eq:eulereval} along with the estimates $a(p;J)=\lambda_f(p)+O\left( \frac{\log p}{\theta_J \log x}\right)$ and
\[
\prod_{p \in I_j}
\left(1+\frac{a(p;J)^2}{2p}-\frac{\lambda_f(p)a(p;J)}{p}+O\left(\frac{1}{p^2} \right) \right)^{-1} \ll 1_{j=0} \cdot (\log x)^{O(1)}+1
\]
we get that the RHS of \eqref{eq:lbdform} equals
\begin{equation} \label{eq:lowerbdcombined}
\begin{split}
&C' x(\log x)^{1/2}\prod_{c< p \le x^{\theta_J}}\left( 1-\frac{\lambda_f(p)^2}{2p}+O\left(\frac{1}{\theta_J\log x} \cdot \frac{ \log p}{p }+\frac{1}{p^2} \right)\right) \\
& \qquad \qquad \qquad \qquad \qquad \times \prod_{0 \le j \le J} \left(1+O\left(\frac{1_{j=0} \cdot (\log x)^{O(1)}+1}{4^{\ell_j}} \right) \right).
\end{split}
\end{equation}
To esimate the second product above note that
\begin{equation} \label{eq:erroreasy}
\prod_{0 \le j \le J}\left(1+O\left(\frac{1_{j=0} \cdot (\log x)^{O(1)}+1}{4^{\ell_j}} \right) \right)=1+O\left(\frac{1}{4^{\ell_J}}\right)=1+O(\eta_2) \ge \frac12,
\end{equation}
since $\eta_2$ is sufficiently small. Also, there exists some constant $C''>0$ such that the
Euler product over $c< p \le x^{\theta_J}$  in \eqref{eq:lowerbdcombined} is 
\[
\ge \prod_{c<p \le x^{\theta_J}} \left(1-\frac{\lambda_f(p)^2}{2p}-C''\left(\frac{1}{\theta_J\log x} \cdot \frac{ \log p}{p }+\frac{1}{p^2} \right) \right)\gg (\log x)^{-1/2},
\]
since $c$ is sufficiently large (so each term in the Euler product is positive).
 Hence, using this and \eqref{eq:erroreasy} we get that \eqref{eq:lowerbdcombined} is $ \gg x (\log x)^{1/2} \cdot (\log x)^{-1/2} = x$, which completes the proof.

\subsection{The proof of Proposition \ref{prop:twisted}}
Let $f$ be a weight $2k$, level $1$, Hecke cusp form. 
For a fundamental discriminant $d$, let
\[
\Lambda(s,f \otimes \chi_d)=\left(\frac{|d|}{2\pi} \right)^s \Gamma(s+\tfrac{2k-1}{2}) L(s,f \otimes \chi_d).
\]
The functional equation for $\Lambda(s, f\otimes \chi_d)$ is given by
\[
\Lambda(s,f \otimes \chi_d)=(-1)^k \tmop{sgn}(d) \Lambda(1-s,f \otimes \chi_d).
\]
Note that the central value vanishes when $(-1)^k d <0$.

For $c >0$ let
\[
W_{z}(x)=\frac{1}{2\pi i} \int_{(c)} \frac{\Gamma(z+s+k-\tfrac12)}{\Gamma(z+k-\tfrac12)} \left(2\pi x\right)^{-s} \, \frac{ds}{s}.
\]
Our starting point in the proof of Proposition \ref{prop:twisted} is the following approximate functional equation for $L(s, f\otimes \chi_d)$. Also, define $W =W_{1/2}$.
\begin{lemma} \label{lem:approxfunc}
Let $f$ be a level $1$, Hecke cusp form with weight $2k$.
For $s\in \mathbb C$ with $0 \le \tmop{Re}(s) \le 1$ and $d$ a fundamental discriminant
\[
\begin{split}
L(s,f \otimes \chi_d)=&\sum_{n \ge 1} \frac{\lambda_f(n) \chi_d(n)}{n^s} W_{s}\left( \frac{n}{|d|}\right)\\
&+(-1)^k \tmop{sgn}(d) \left( \frac{ |d|}{2\pi} \right)^{1-2s} \frac{\Gamma(1-s)}{\Gamma(s)} \sum_{n \ge 1} \frac{\lambda_f(n) \chi_d(n)}{n^{1-s
}} W_{1-s}\left( \frac{n}{|d|}\right).
\end{split}
\]
Moreover, the function $W=W_{1/2}$ satisfies $ \xi^j W^{(j)}(\xi) \ll |\xi|^{-A}$ for any $A \ge 1$ and $W(\xi)=1+O(\xi^{k-\varepsilon})$ as $\xi \rightarrow 0$.
\end{lemma}

\begin{proof} See Lemma 5 of \cite{RadziwillSound} and Lemma 2.1 of \cite{SoundYoung}. 
\end{proof}

Since we will sum over fundamental discriminants we introduce a new parameter $Y$ with $1 \le Y \le x$ to be chosen later.
Also, write $F(\xi;x,y)=\phi\left(\frac{\xi}{x}\right)W\left(\frac{y}{\xi}\right) $. Applying Lemma \ref{lem:approxfunc} the LHS of \eqref{eq:twist} equals
\begin{equation} \label{eq:split}
\begin{split}
&2 \sum_{ (m,2)=1 }  \sum_{a^2 | m} \mu(a) \sum_{n \ge 1} \frac{\lambda_f(n) }{\sqrt{n}} \left( \frac{ 8m}{nu}\right)F\left( (-1)^k 8m; x,n \right)\\
&=2 \left( \sum_{\substack{a \le Y \\ (a,2u)=1}}+\sum_{\substack{a>Y \\ (a,2u)=1}}\right) \mu(a) \sum_{\substack{(m,2)=1 \\ a^2 |m}}    \sum_{n \ge 1} \frac{\lambda_f(n)}{\sqrt{n}} \left( \frac{ 8m}{nu}\right) F\left( (-1)^k 8m; x,n\right) .
\end{split}
\end{equation}
\paragraph{\textbf{The terms with $a>Y$}.}
Write $m=r^2e$ where $e$ is square-free, and note that since in the sum in \eqref{eq:split} $a^2|m$ it follows $a|r$. So the second sum in \eqref{eq:split} equals
\begin{equation} \label{eq:largea1}
\begin{split} 
& 2 \sum_{(r,2u)=1} \sum_{\substack{a>Y \\ a|r}} \mu(a) \sum_{\substack{e \text{ is } \square-\text{free} \\ (e,2)=1}}\sum_{(n,r)=1} \frac{\lambda_f(n)}{\sqrt{n}} \left( \frac{ 8e}{nu}\right) F\left( (-1)^k 8 r^2 e; x,n\right) \\
& \ll  \sum_{|r | \le x^{1/2+\varepsilon}} \sum_{\substack{a>Y \\ a|r, (a,2)=1}}  \sum_{\substack{ 1 \le (-1)^k e \le x^{1+\varepsilon}/r^2 \\ e \text{ is } \square-\text{free} \\ (e,2)=1}} \left| \sum_{(n,r)=1} \frac{\lambda_f(n)}{\sqrt{n}} \left( \frac{ 8e}{n}\right) W\left(\frac{n}{8r^2|e|}\right) \right| +\frac{1}{x},
\end{split}
\end{equation}
where in the second line we have used the rapid decay of $W$.
Using the definition of $W$ and for $\tmop{Re}(s)>1$, writing $L(s,f\otimes \chi_{8e})=\prod_p L_p(s)$ where $L_p(s)=\left(1-\frac{\lambda_f(p)\left(\frac{8e}{p}\right)}{p^s}+\frac{\left(\frac{8e}{p}\right)^2}{p^{2s}}\right)^{-1}$ we get for $c>1/2$ that
\begin{equation} \label{eq:mellin}
\begin{split}
\sum_{(n,r)=1} \frac{\lambda_f(n)}{\sqrt{n}} \left( \frac{ 8e}{n}\right) W\left(\frac{n}{8r^2|e|}\right)
=& \frac{1}{2\pi i}  \int_{(c)}
\frac{\Gamma(s+k)}{\Gamma(k)} \left( \frac{8r^2e}{2\pi}\right)^s \frac{L(s+\frac12,f \otimes \chi_{8e})}{\prod_{p | r} L_p(s+\frac12)} \, \frac{ds}{s}.
\end{split}
\end{equation}
The integrand is holomorphic for $1/\log x \le \tmop{Re}(s) \le 2$ and in this region bounded by (note $r^2e \le x^3$)
\[
\ll \prod_{p|r}\left(1+\frac{O(1)}{\sqrt{p}} \right) |\Gamma(s+k)| |L(s+\tfrac12, f\otimes \chi_{8e})| \cdot \frac{1}{|s|} \ll x^{\varepsilon}  e^{-|s|} \cdot |L(s+\tfrac12, f\otimes \chi_{8e})|.
\]
Hence, shifting contours on the RHS of \eqref{eq:mellin} to $\tmop{Re}(s)=1/\log x$ and applying the above estimate we get that the LHS of \eqref{eq:mellin} is bounded by
\[
\ll x^{\varepsilon} \int_{(1/\log x)} e^{-|s|} |L(s+\tfrac12, f \otimes \chi_{8e})| \, ds.
\]
Also note, that by Cauchy-Schwarz and Corollary 2.5 of \cite{SoundYoung} (which follows from Heath-Brown's result \cite{HB})
\[
\sum_{\substack{ 1 \le (-1)^k e \le x^{1+\varepsilon}/r^2 \\ e \text{ is } \square-\text{free} \\ (e,2)=1}} |L(s+\tfrac12, f \otimes \chi_{8e})| \ll \frac{x^{1+\varepsilon}}{r^2} (1+|\tmop{Im}(s)|)^{1/2+\varepsilon},
\]
for $\tmop{Re}(s) \ge 0$.
Applying this estimate in \eqref{eq:largea1} it follows that the second sum in \eqref{eq:split} is bounded by
\begin{equation} \label{eq:largeabd}
\ll x^{1+\varepsilon} \sum_{|r| \le x^{1/2+\varepsilon}} \sum_{\substack{a>Y \\ a| r}} \frac{1}{r^2} \ll \frac{x^{1+\varepsilon}}{Y}.
\end{equation}

\vspace{.1 in}

\paragraph{\textbf{The terms with $a<Y$: preliminary lemmas}}
It remains to estimate the first sum on the LHS of \eqref{eq:split}. This will be done by applying Poisson summation to the character sum, as developed in \cite{SoundNonvanishing}. 

Let $\ell \in \mathbb Z, n \in \mathbb N$. Define
\[
G_{\ell}(n)=\left(\frac{1-i}{2}+\left( \frac{-1}{n}\right)\frac{1+i}{2} \right) \sum_{a \pmod{n}} \left( \frac{a}{n}\right) e\left( \frac{a\ell}{n}\right).
\]
In Lemma 2.3 of \cite{SoundNonvanishing} it is shown that $G_{\ell}$ is a multiplicative function. Moreover, $G_0(n)=\varphi(n)$ if $n$ is a square and is identically zero otherwise. Also, for $p^{\alpha} || \ell$, $\ell \neq 0$
\begin{equation} \label{eq:def}
G_{\ell}(p^{\beta})=
\begin{cases}
0 & \text{if } \, \beta \le \alpha \text{ is odd},\\
\varphi(p^{\beta}) & \text{if } \,\beta \le \alpha \text{ is even}, \\
-p^{\alpha} & \text{if } \, \beta = \alpha+1 \text{ is even},\\
\left(\frac{\ell p^{-\alpha}}{p} \right)p^{\alpha} \sqrt{p} & \text{if } \, \beta = \alpha+1 \text{ is odd},\\
0 & \text{if } \, \beta \ge \alpha+2.
\end{cases}
\end{equation}

\begin{lemma} \label{lem:poisson} Let $F$ be a Schwartz function. Then for any odd integer $n$
\[
\sum_{(d,2)=1} \left( \frac{d}{n} \right) F\left( d\right)= \frac{1}{2n} \left( \frac{2}{n} \right) \sum_{\ell} (-1)^{\ell} G_{\ell}(n) \widetilde F\left( \frac{\ell }{2n} \right),
\]
where
\[
\widetilde F(\lambda)=\int_{\mathbb R}(\cos(2\pi \lambda \xi)+\sin(2\pi \lambda \xi))F(\xi) \, d\xi.
\]
\end{lemma}
\begin{proof}
This is established in the proof of Lemma 2.6 of \cite{SoundNonvanishing}, in particular see the last equation of the proof.
\end{proof}

Using Lemma \ref{lem:approxfunc} it follows that
$F( \cdot; x,n)$ is a Schwartz function, since $W$ and $\phi$ and their derivatives decay rapidly. Hence,  applying Lemma \ref{lem:poisson} the first sum on the RHS of \eqref{eq:split} equals (note $u$ is odd)

\begin{equation} \label{eq:poissondone}
\begin{split}
&2\sum_{\substack{a \le Y \\ (a,2u)=1}} \sum_{(n,a)=1}  \frac{\lambda_f(n)}{\sqrt{n}} \left(\frac{8}{nu} \right) \sum_{(m,2)=1} \left(\frac{m}{nu} \right) F\left(m;\frac{(-1)^k x}{8a^2},\frac{(-1)^kn}{8a^2} \right) \\
&=2\sum_{\substack{a \le Y \\ (a,2u)=1}} \mu(a) \sum_{(n,a)=1} \frac{\lambda_f(n)}{\sqrt{n}} \cdot \frac{1}{2nu} \left( \frac{16}{nu} \right) \sum_{\ell \in \mathbb Z} (-1)^{\ell} G_{\ell}(nu) \widetilde F\left( \frac{\ell}{2nu}; \frac{ x}{8a^2},\frac{ n}{8a^2}  \right),
\end{split}
\end{equation}
where
\[
 \widetilde F\left( \lambda; x, y\right)=\int_{\mathbb R}(\cos(2\pi \lambda \xi)+(-1)^k\sin(2\pi \lambda \xi))F(\xi;x,y) \, d\xi.
\]
Note that for odd $k$ after applying Lemma \ref{lem:poisson} in \eqref{eq:poissondone}  we also made the change of variables $\xi \rightarrow -\xi$.

Before proceeding we require several estimates for $\widetilde F$. The first such result is a basic estimate on the rate of decay of $\widetilde F$.

\begin{lemma} \label{lem:fourierbd}
We have
\[
\widetilde F(\lambda;x,y) \ll x \min\left( \left( \frac{x}{y}\right)^A, \frac{1}{|y\lambda|^A} \right)
\]
for any $A \ge 1$.
\end{lemma}
\begin{proof}
Using the rapid decay of $W$ it follows that
\[
 W\left( \frac{y}{x \xi}\right) \ll \left|\frac{x \xi}{y} \right|^A
\]
so that
\[
\widetilde F(\lambda;x,y) \ll x \left(\frac{x}{y} \right)^A \int_{\mathbb R} |\xi|^A |\phi(\xi)| \, d\xi \ll x \left(\frac{x}{y} \right)^A.
\]

Next, write $\mathcal W(\xi)=W\left( \frac{y}{x \xi}\right)\phi(\xi)$ and note that for any $A,B,C \ge 0$, with $A \ge B$, and $|\xi| \gg 1$
\[
\left( \frac{y}{x}\right)^B W^{(C)} \left( \frac{y}{x \xi}\right) \ll \left(\frac{x}{y} \right)^{A}|\xi|^{A-B}
\]
so using this and the fact that  $\phi$ has compact support it follows that
\[
\mathcal W^{(A)}(\xi) \ll \left(\frac{x}{y} \right)^{A} \cdot \frac{1}{1+|\xi|^2}.
\]
Hence, integrating by parts and using the above bound it follows that
\[
\widetilde F(\lambda;x,y) \ll x \cdot \frac{1}{|x\lambda|^A} \int_{\mathbb R}  \big|\mathcal W^{(A)}(\xi) \big|\, d \xi  \ll x \cdot \frac{1}{|y\lambda|^A}.
\]
\end{proof}

We also require the following information about the Mellin transform of $\widetilde F$. Let
\[
\Phi(s)=\int_{0}^{\infty} \phi(t) t^{-s} \, ds. 
\]
\begin{lemma} \label{lem:continuation}
For $\ell \in \mathbb Z$ with $\ell \neq 0$, let
\[
\breve{F}(s;\ell,u, a^2)=\int_0^{\infty} \widetilde F\left( \frac{\ell}{2tu}; \frac{ x}{8a^2}, \frac{t}{8a^2}\right) \, t^{s-1} \, dt.
\]
Then for $\tmop{Re}(s)>0$
\begin{equation} \label{eq:breveF}
\breve{F}(s;\ell,u, a^2)= \frac{x^{1+s}}{8a^2} \Phi(-s) \int_{0}^{\infty} W\bigg( \frac{1}{y}\bigg) \left(\cos\left(\pi y \cdot \frac{\ell}{8ua^2} \right)+(-1)^k\sin\left(\pi y \cdot \frac{\ell}{8ua^2} \right)\right)\, \frac{dy}{y^{s+1}}.
\end{equation}
Moreover, the function $\breve{F}(s;\ell,u, a^2)$ extends to an entire function in the half-plane $\tmop{Re}(s)\ge -1+\varepsilon$ and in this region
\[
\breve{F}(s;\ell,u, a^2) \ll \frac{x^{1+\tmop{Re}(s)}}{a^2(1+|s|)^A}\left(\left(\frac{|\ell|}{ua^2}\right)^{\tmop{Re}(s)}+1\right)
\]
for any $A \ge 1$.
\end{lemma}

\begin{proof}
Changing the order of integration and making a change of variables $x\xi/t \rightarrow \xi$ in the integral over $t$ it follows that
$\breve{F}(s;\ell,u, a^2)$ equals
\[
\begin{split}
& \frac{x}{8a^2}\int_0^{\infty}  \int_{\mathbb R}   \left(\cos\left(\pi \frac{\ell}{8ua^2} \cdot \frac{x \xi}{t} \right)+(-1)^k\sin\left(\pi \frac{\ell}{8ua^2} \cdot \frac{x \xi}{t} \right)\right) W\left( \frac{t}{x \xi}\right) \phi(\xi) t^{s-1}   \, d\xi \, dt\\
=&\frac{x^{1+s}}{8a^2}  \int_{\mathbb R}\int_0^{\infty} \phi(\xi)  \xi^s \cdot W\left( \frac{1}{t}\right) \left(\cos\left(\pi t \cdot \frac{\ell}{8ua^2}  \right)+(-1)^k\sin\left(\pi t \cdot \frac{\ell}{8ua^2}  \right)\right) \frac{dt}{t^{s+1}}d\xi ,
\end{split}
\]
which establishes the first claim.

The function 
\[
w(s)=\int_{0}^{\infty} W\bigg( \frac{1}{y}\bigg) \left(\cos\left(\pi y \cdot \frac{\ell}{8ua^2} \right)+(-1)^k\sin\left(\pi y \cdot \frac{\ell}{8ua^2} \right)\right)\, \frac{dy}{y^{s+1}}
\]
is holomorphic in the region $\tmop{Re}(s) \ge \varepsilon$. Write $w(s)=I_1(s)+I_2(s)$ where $I_1$ is the portion of the integral over $[0, \tfrac{8au^2}{|\ell|}]$ and $I_2$ is the rest. Due to the rapid decay of $W$, $I_1(s)$ is holomorphic in the region $\tmop{Re}(s) \ge -1$ and in this region
\begin{equation} \label{eq:I1bd}
I_1(s) \ll \int_0^{\frac{8ua^2}{|\ell|}} |W(1/y)| \, y^{-\tmop{Re}(s)-1} dy \ll \left(\frac{|\ell|}{ua^2}\right)^{\tmop{Re}(s)}+1.
\end{equation}

Next, write
\begin{equation} \label{eq:splitint}
I_2(s) = \left(\frac{|\ell|}{8ua^2} \right)^s \int_1^{\infty} \left(\cos\left( \pi y  \right)+(-1)^k\tmop{sgn}(\ell)\cdot\sin\left(\pi y \right)\right) \, \frac{dy}{y^{s+1}} +I_3(s).
\end{equation}
The integral $I_3(s)$ is holomorphic in $\tmop{Re}(s) \ge -1$ and uniformly in this region we have by Lemma \ref{lem:approxfunc}, that $W(1/y)=1+O(y^{-k+\varepsilon})$  we get
\[
I_3(s) \ll \int_{\frac{8au^2}{|\ell|}}^{\infty}\left|1-W\left(\frac{1}{y} \right) \right| \, \frac{dt}{y^{\tmop{Re}(s)+1}} \ll \left(\frac{|\ell|}{u a^2} \right)^{\tmop{Re}(s)}+1.
\]
The first integral on the RHS of \eqref{eq:splitint} can be analytically continued to $\tmop{Re}(s)>-1$ by integrating by parts. This provides the analytic continuation of $I_2(s)$ to $\tmop{Re}(s)\ge-1+\varepsilon$, and shows that in this region 
\[
I_2(s) \ll (|s|+1)\left( \left(\frac{|\ell|}{ua^2} \right)^{\tmop{Re(s)}}+1\right).
\]
Hence applying this estimate along with \eqref{eq:I1bd} in \eqref{eq:breveF} and noting $\Phi$ decays rapidly establishes the claim.
\end{proof}

\paragraph{\textbf{The terms with $a<Y$: main term analysis i.e. $\ell=0$}} Recall that $G_0(n)=\varphi(n)$ if $n=\square$ and is zero otherwise. So using Lemma \ref{lem:fourierbd}
the term with $\ell=0$ in \eqref{eq:poissondone} equals
\begin{equation} \label{eq:ellzero}
\begin{split}
& \sum_{(a,2u)=1} \mu(a) \sum_{\substack{(n,2a)=1 \\ nu=\square}} \frac{\lambda_f(n)}{\sqrt{n}} \frac{\varphi(nu)}{nu} \widetilde F\left(0; \frac{x}{8a^2},\frac{n}{8a^2} \right)+O\left( \frac{x^{1+\varepsilon}}{Y}\right) \\
& =\frac{x}{8}
\frac{1}{2\pi i} \int_{(c)} \left( \frac{x}{2\pi} \right)^s \frac{\Gamma(s+k)}{\Gamma(k)} \left( \sum_{\substack{a  \\(a,2u)=1}} \frac{\mu(a)}{a^2} \sum_{\substack{(n,2a)=1 \\ nu=\square}} \frac{\lambda_f(n)}{n^{s+1/2}} \frac{\varphi(nu)}{nu} \right) \, \Phi(-s) \frac{ds}{s} +O\left( \frac{x^{1+\varepsilon}}{Y}\right).
\end{split}
\end{equation}

We now evaluate the inner sum on the RHS \eqref{eq:ellzero} (note $\tmop{Re}(s)=c>0$). Since $nu=\square$ write $n=e^2 u_1$, (recall $u=u_1u_2^2$) so  
\begin{equation} \label{eq:switch}
\begin{split}
\sum_{(a,2u)=1} \frac{\mu(a)}{a^2} \sum_{\substack{(n,2a)=1 \\ nu=\square}} \frac{\lambda_f(n)}{n^{s+1/2}} \frac{\varphi(nu)}{nu}=&  \sum_{(e,2)=1} \frac{\lambda_f(u_1 e^2)}{(u_1e^2)^{s+1/2}} \frac{\varphi(e^2u_1u)}{e^2u_1u} \sum_{(a,2eu)=1} \frac{\mu(a)}{a^2} \\
=&\frac{4}{3\zeta(2) u_1^{s+\frac12}}  \sum_{(e,2)=1} \frac{\lambda_f(u_1 e^2)}{e^{2s+1}} \prod_{p|eu}\frac{1-\frac1p}{1-\frac{1}{p^2}}.
\end{split}
\end{equation}
The sum on the RHS can be expressed as an Euler product as follows. Let $r(n)=\prod_{p|n} \frac{1-\frac1p}{1-\frac{1}{p^2}}$. Also, let
\[
\widetilde \sigma_p(z;\alpha,\beta)=\sum_{j=0}^{\infty} \frac{\lambda_f(p^{2j+\alpha})}{p^{jz}} r(p^{j+\beta})
\]
and
\[
\mathcal G(z;u)=  \prod_{p|u_1} \frac{\widetilde \sigma_p(z;1,1)}{\widetilde \sigma_p(z;0,0)} \prod_{p|\frac{u_2}{(u_1,u_2)}} \frac{\widetilde \sigma_p(z;0,1)}{\widetilde \sigma_p(z;0,0)}.
\]
It follows that,
\begin{equation} \label{eq:factor}
 \sum_{(e,2)=1} \frac{\lambda_f(u_1 e^2)}{e^{2s+1}} \prod_{p|eu}\frac{1-\frac1p}{1-\frac{1}{p^2}}= L(1+2s, \tmop{sym}^2 f)\frac{\mathcal G(1+2s;u)}{\widetilde \sigma_2(1+2s;0,0)}.
\end{equation}

Also, that $\frac{\mathcal G(1+2s;u)}{\widetilde \sigma_2(1+2s;0,0)} \ll u^{\varepsilon} $ for $\tmop{Re}(s) \ge -\frac14$ and is analytic in this region. Hence, applying \eqref{eq:switch} and \eqref{eq:factor} in the RHS of \eqref{eq:ellzero} then shifting contours  $\tmop{Re}(s)=-\tfrac14+\varepsilon$, collecting the residue at $s=0$ we see that the RHS of \eqref{eq:ellzero} equals
\begin{equation} \label{eq:conclusionellzero}
 \frac{x}{\pi^2 u_1^{1/2}} \cdot L(1, \tmop{sym}^2 f)\frac{\mathcal G(1;u)}{\widetilde \sigma_2(1;0,0)} \cdot \Phi(0) +O\left(u^{\varepsilon} x^{3/4+\varepsilon} \right)+O\left( \frac{x^{1+\varepsilon}}{Y}\right).
\end{equation}
Finally, note $\mathcal G(1;\cdot)$ is a multiplicative function satisfying $\mathcal G(1;p^{2j+1})=\lambda_f(p)+O(1/p)$ and $G(1;p^{2j})=1+O(1/p)$. Set $\vartheta(n)=\mathcal G(1;n)$.

\vspace{.1 in}

\paragraph{\textbf{The terms with $a<Y$: Off-diagonal analysis i.e. $\ell \neq 0$}}

In \eqref{eq:poissondone}
it remains to bound
\begin{equation} \label{eq:offterms}
\sum_{\ell \neq 0}\sum_{\substack{ a \le Y \\ (a,2u)=1}} \bigg| \sum_{(n,2a)=1} \frac{\lambda_f(n)}{n^{1/2}} \left(\frac{4}{n} \right) \frac{G_{\ell}(nu)}{nu} \widetilde F\left(\frac{\ell}{2nu}; \frac{ x}{8a^2}, \frac{n}{8a^2} \right)\bigg|.
\end{equation}
First, note that by Lemma \ref{lem:fourierbd} the contribution from the terms with $|\ell| \ge Y^2 u x^{\varepsilon}$ to \eqref{eq:offterms} is bounded by
\begin{equation} \label{eq:trivellbd}
\ll
\sum_{|\ell| \ge Y^2 u x^{\varepsilon}} \sum_{a \le Y} \frac{x}{a^2}\left(\sum_{n \le |\ell|^{\varepsilon} x} \left( \frac{a^2 u}{ |\ell|} \right)^A+\sum_{n > |\ell|^{\varepsilon} x} \left( \frac{x}{n} \right)^A  \right)  \ll \frac{1}{x}
\end{equation}
where $A$ has been chosen sufficiently large with respect to $\varepsilon$.

It remains to estimate the terms in \eqref{eq:offterms} with $|\ell| \le Y^2 u x^{\varepsilon} $.
Let $\psi_{4\ell}(n)=\left( \frac{4\ell}{n}\right)$ and note that $\psi_{4\ell}$ is a character of modulus at most $4|\ell|$.
Using that $G_{\ell}$ is a multiplicative function, and using \eqref{eq:def}  we can write
\begin{equation} \label{eq:factorsum}
\sum_{(n,2a)=1} \frac{\lambda_f(n)}{n^{s}} \left(\frac{4}{n} \right) \frac{G_{\ell}(nu)}{ u\sqrt{n}}=L(s,f \otimes \psi_{4\ell}) R(s;\ell,u,a), \qquad \tmop{Re}(s)>1,
\end{equation}
where
\[
R(s;\ell,u,a)=\prod_{p} R_p(s;\ell,u,a).
\]
It follows that
for $(p,4au\ell)=1$
\[
R_p(s;\ell,u,a)=\left(1+\frac{\lambda_f(p)}{p^s}\left( \frac{4\ell}{p} \right)\right)\cdot \left(1- \frac{\lambda_f(p)}{p^s} \left(\frac{4\ell}{p}\right) +\frac{1}{p^{2s}}\right) =1+O\left( \frac{1}{p^{2\tmop{Re}(s)}}\right).
\]
For $p |au \ell$ and $\tmop{Re}(s) \ge \tfrac12+\varepsilon$, writing $p^{\theta} || u$, we have that
\[
\begin{split}
|R_p(s;\ell,u,a)|=&\bigg|\left(1+O\left(\frac{1}{p^{1/2}} \right) \right)
\cdot \left(\frac{G_{\ell}(p^{\theta})}{p^{\theta}}+O\left(\sum_{j \ge 1} \frac{p^{j+\theta}}{p^{j(\frac12+\varepsilon)} p^{j/2+\theta}} \right) \right)\bigg| \\
\le & \bigg(1+O\left( \frac{1}{p^{\varepsilon}}\right) \bigg).
\end{split}
\]
Also $R_2(s;\ell,u,a) \ll 1$ for $\tmop{Re}(s) \ge \frac12$.
Hence, for $0<a, u, |\ell| \le x^2$ and $\tmop{Re}(s) \ge \tfrac12+\varepsilon$
\begin{equation} \label{eq:Rbd}
R(s;\ell,u,a)\ll \prod_{p|au\ell} \left(1+\frac{O(1)}{p^{\varepsilon}} \right) \ll x^{\varepsilon}.
\end{equation}
So $R(\cdot; \ell, u,a)$ is an absolutely convergent Euler product and thereby defines a holomorphic function in the half-plane $\tmop{Re}(s)\ge \frac12+\varepsilon$. 
Hence, applying Mellin inversion and \eqref{eq:factorsum} we get for $c>0$ that
\begin{equation} \label{eq:mellinfactored}
\begin{split}
&\sum_{(n,2a)=1} \frac{\lambda_f(n)}{n} \left(\frac{4}{n} \right) \frac{G_{\ell}(nu)}{u\sqrt{n}} \widetilde F\left(\frac{\ell}{2nu}; \frac{ x}{8a^2}, \frac{n}{8a^2} \right)\\
&\qquad \qquad =\frac{1}{2\pi i} \int_{(c)} \breve{F}(s; \ell,u, a^2) L(1+s,f \otimes \psi_{4\ell}) R(1+s;\ell,u,a)\, ds .
\end{split}
\end{equation}

Using Lemma \ref{lem:continuation} and \eqref{eq:Rbd} we can shift contours in the integral above to $\tmop{Re}(s)=-\tfrac12+\varepsilon$, since the integrand is holomorphic in this region.
Using the convexity bound 
\[
L(\tfrac12+\varepsilon+it, f \otimes \psi_{4\ell}) \ll (1+|t|)^{1/2+\varepsilon} |\ell|^{1/2+\varepsilon}
\]
along with Lemma \ref{lem:continuation} and \eqref{eq:Rbd} it follows that for $|\ell|,u,a \le x$
\begin{equation} \label{eq:shiftedintbd}
\frac{1}{2\pi i} \int_{(-\frac{1}{2}+\varepsilon)} \breve{F}(s; \ell,u, a^2) L(1+s,f \otimes \psi_{4\ell}) R(1+s;\ell,u,a)\, ds ds \ll \frac{x^{\frac12+\varepsilon}}{a^2} \cdot  \left(\left( \frac{ua^2}{|\ell|} \right)^{1/2}+1 \right) |\ell|^{1/2}. 
\end{equation}
Applying \eqref{eq:mellinfactored} and \eqref{eq:shiftedintbd}
in \eqref{eq:offterms} the terms with $|\ell| \le Y^2 x^{\varepsilon} u$ in \eqref{eq:offterms} are bounded by
\begin{equation} \label{eq:finaloffbd}
x^{\frac12+\varepsilon} \sum_{1 \le |\ell| \le Y^2 x^{\varepsilon} u}\sum_{a \le Y}\frac{1}{a^2}\left( \left( \frac{ua^2}{|\ell|}\right)^{1/2}+1\right) |\ell|^{1/2}
\ll x^{\frac12+\varepsilon} Y^3 u^{3/2}.
\end{equation}

\paragraph{\textbf{Completion of the proof of Proposition \ref{prop:twisted}}}
Applying \eqref{eq:conclusionellzero}, \eqref{eq:trivellbd}, \eqref{eq:finaloffbd} in \eqref{eq:poissondone} and then using the resulting formula along with \eqref{eq:largeabd} in \eqref{eq:split} it follows that
\begin{equation} \label{eq:Cdef}
\begin{split}
&\sum_{\substack{m \\ 2m \text{ is } \square-\text{free}}} L(\tfrac12, f \otimes \chi_{8m}) \chi_{8m}(u) \, \phi\left( \frac{(-1)^k 8m}{x}\right)\\
& \qquad \qquad =
\frac{x}{\pi^2 u_1^{1/2}} \cdot L(1, \tmop{sym}^2 f)\frac{\mathcal G(1;u)}{\widetilde \sigma_2(1;0,0)} \cdot \Phi(0)+O\left( \frac{x^{1+\varepsilon}}{Y}+x^{\frac12+\varepsilon} Y^3 u^{3/2}\right).
\end{split}
\end{equation}
Choosing $Y=x^{1/8}u^{-3/8}$ and recalling the estimates given for $\mathcal G(1;\cdot)$ after \eqref{eq:conclusionellzero} completes the proof.

\section{The Proof of Proposition \ref{mainprop}}

In order to establish Proposition \ref{mainprop} we will need the following variant of the shifted convolution problem for coefficients of half-integral weight forms. 

\begin{proposition} \label{scpprop}
  Uniformly in $1 \leq \Delta \leq X^{\frac{1}{52}}$, $0 < |h| < X^{\frac{1}{2}}$ and $(v, \Delta) = 1$, we have for every $\varepsilon > 0$
  \begin{equation} \label{eq:mainscp}
  \frac{1}{X^{k - 1/2}} \Big | \sum_{n} c(n) n^{\frac{k - 1/2}{2}} e \Big ( \frac{n v}{\Delta} \Big ) e^{- \frac{2\pi n}{X}} c(n + h) (n + h)^{\frac{k - 1/2}{2}} e^{- \frac{2\pi (n + h)}{X} } \Big |  \ll X^{1 - \frac{1}{104} + \varepsilon}. 
  \end{equation}
\end{proposition}

\subsection{The shifted convolution problem}

We begin with a proof of Proposition \ref{scpprop}. The proof is based on the by now standard combination of the circle method and modularity. Recall that a weight $k+\tfrac12$ modular form (with trivial character) transforms under $\Gamma_0(4)$ in the following way
\begin{equation} \label{eq:transformation}
  g ( \gamma z) = \nu(\gamma) j_{\gamma}(z)^{k + \frac 12}g(z), \qquad \forall \gamma \in \Gamma_0(4),
\end{equation}
where for $\gamma=
\begin{pmatrix}
a & b \\
c& d
\end{pmatrix}$ we set
$
j_{\gamma}(z)=cz+d$ and  $\nu(\gamma)=  \Big( \frac{c}{d}\Big) \overline {\varepsilon_d}$, where $\Big( \frac{\cdot}{\cdot}\Big)$ denotes the quadratic residue symbol in the sense of Shimura \cite{Shimura1} (see \textit{Notation} 3)  and
  $$
  \varepsilon_{d} := \begin{cases}
    1 & \text{ if } d \equiv 1 \pmod{4}, \\
    i^{2k + 1} & \text{ if } d \equiv - 1 \pmod{4}.
    \end{cases}
    $$
    Also, we record the following estimate for the Fourier coefficients of $g$
\begin{equation} \label{eq:Parseval}
\sum_{n \leq X} |c(n)|^2 \ll X
\end{equation}
(see \cite[Theorem 5.1]{IwaniecBook} and use partial summation). This implies that
\begin{equation} \label{eq:Fourierbd}
|c(n)| \ll n^{1/2+\varepsilon}.
\end{equation}

We record below a few standard lemmas. 

\begin{lemma} \label{le:circle}

  Let $\eta \in (0, 1]$ and $Q \geq 1$ be given. Let $Q^{\eta / 2} \geq \Delta \geq 1$ be an integer.
  Let $I_{d/q}(\cdot)$ denote the indicator function of the interval $[d / q - Q^{- 2 + \eta}; d / q + Q^{-2 + \eta}]$. Let $\mathcal{Q}$ denote the set of integers $n \in [Q, 2Q]$ that can be written as $4 \Delta r$ with $r \equiv 1 \pmod{4}$ prime. For $\alpha \in \mathbb R$, put
  $$
  \widetilde{I}(\alpha) = \frac{Q^{2 - \eta}}{2 L} \sum_{q \in \mathcal{Q}} \, \, \sumstar_{d \pamod q} I_{d/q}(\alpha)  \ \quad \text{ and } \quad \ L = \sum_{q \in \mathcal{Q}} \varphi(q).
  $$
  Finally let $I(\cdot)$ denote the indicator function of the interval $[0,1]$. 
  Then, for every $\varepsilon > 0$
  $$
  \int_{\mathbb R} (I(\alpha) - \widetilde{I}(\alpha))^2 d \alpha \ll_{\varepsilon} Q^{-\eta / 2 + \varepsilon}. 
  $$
  \end{lemma}
\begin{proof}
  This is a consequence of a result of Jutila (see \cite{Jutila} or \cite[Proposition 2]{Harcos}). In the notation of \cite[Proposition 2]{Harcos} we specialize to $\delta = Q^{-2 + \eta}$ and notice that $L \asymp \frac{\varphi(4 \Delta)}{\Delta^2} \cdot \frac{Q^2}{\log Q}$.
\end{proof}



Crucial to our analysis is the following consequence of the modularity of $g$. 

\begin{lemma} \label{le:modularity}
  Let $g$ be a cusp form of weight $k + \tfrac 12$ of level $4$ where $k \geq 2$ is an integer. 
    Let $(d,q) = 1$ and $(v,\Delta) = 1$ be two pairs of co-prime integers. Suppose that $q = 4 \Delta r$ with $r > 4 \Delta$ a prime congruent to $1 \pmod{4}$.   Write $d = r b + 4 \Delta \ell$ with $(b, 4 \Delta) = 1$ and $(\ell, r) = 1$. Let $\Delta^{\star} = \Delta / (4 v + b, \Delta)$ and $(4v + b)^{\star} = (4 v + b) / (4v + b, \Delta)$.  
  Then, for any $\beta \in \mathbb{R}$ and any real $X \geq 1$
  \begin{align*} 
  g \Big ( \frac{d}{q} + \frac{v}{\Delta} & + \beta + \frac{i}{X} \Big ) = \overline{\varepsilon_{(4 v + b)^{\star}}}\chi_{4 \Delta^{\star}} ( (4 v + b)^{\star} r ) \chi_{r} ( 4 \Delta^{\star} \ell ) \\ & \times  \Big ( \frac{X}{4 \Delta^{\star} r} \Big )^{k + \frac 12} \Big ( \frac{i}{1 - i \beta X} \Big )^{k + \frac 12}g \Big ( - \frac{\overline{r^2 (4 v + b)^{\star}}}{4 \Delta^{\star}} - \frac{\overline{16 (\Delta^{\star})^2 \ell}}{r} + \frac{i}{(4 \Delta^{\star} r)^2 ( \frac{1}{X} - i \beta)} \Big ),
  \end{align*}
  where $\chi_{t}(\cdot)$ denotes a real Dirichlet character of modulus $t$. Finally whenever we write $\frac{\overline{a}}{q}$ we denote by $\overline{a}$ an integer such that $\overline{a} a \equiv 1 \pmod{q}$. 
\end{lemma}
\begin{remark}
By inspection of the proof below, the conclusion of the lemma also holds when $v=0$ and we will use this later.
\end{remark}
\begin{proof} Since $q = 4 \Delta r$ we can write
  $$
  \frac{d}{q} + \frac{v}{\Delta} = \frac{d + 4 v r}{4 \Delta r} .
  $$
  Notice that $(d + 4 v r, 4 \Delta r) = (d + 4 v r, \Delta)$ since $(d, 4 r) = 1$ by assumption. Therefore the above is equal to
  \begin{equation} \label{eq:fractions} 
  \frac{(4 v + b)^{\star} r + 4 \Delta^{\star} \ell}{4 \Delta^{\star} r}.
  \end{equation}
  Throughout set $w := (4 v + b)^{\star} r + 4 \Delta^{\star} \ell$. 
  We now consider
  $$
  z = - \frac{\overline{w}}{4 \Delta^{\star} r} + \frac{i}{(4 \Delta^{\star} r)^2 (\frac{1}{X} - i \beta)} 
  $$
  and the matrix
$$
\gamma = \begin{pmatrix}
  w & \star \\
  4 \Delta^{\star} r & \overline{w}
  \end{pmatrix} \in \Gamma_0(4).
  $$
  A slightly tedious computation using \eqref{eq:fractions} reveals that 
  \begin{equation} \label{equa:one}
  \gamma z = \Big ( \frac{d}{q} + \beta \Big ) + \frac{v}{\Delta} + \frac{i}{X}. 
  \end{equation}
  We also find that
  \begin{equation} \label{equa:two}
  j_{\gamma}(z) = \frac{i}{\frac{4 \Delta^{\star} r}{X} - i (4 \Delta^{\star} r) \beta} = \frac{1}{4 \Delta^{\star} r} \cdot \frac{i X}{1 - i X \beta} 
  \end{equation}
  and that
  \begin{equation} \label{equa:three}
  \nu(\gamma) = \overline{\varepsilon_{(4 v + b)^{\star}}} \cdot \Big ( \frac{4 \Delta^{\star} r}{\overline{w}} \Big ) ,
  \end{equation}
  where $\Big ( \frac{\cdot}{\cdot} \Big )$ denotes the extended quadratic residue symbol in the sense of Shimura. In particular since $4 \Delta^{\star} r$ is divisible by four, this extended quadratic residue symbol coincides with $\chi_{4 \Delta^{\star} r}(w)$, a real Dirichlet character of modulus $4 \Delta^{\star} r$, so that $\chi_{4 \Delta^{\star} r}(\overline{w}) = \chi_{4 \Delta^{\star} r}(w)$. Moreover by multiplicativity of the Jacobi symbol
  $$
  \chi_{4 \Delta^{\star} r}(w) = \Big ( \frac{4 \Delta^{\star} r}{w} \Big ) = \Big ( \frac{4 \Delta^{\star}}{w} \Big ) \Big ( \frac{r}{w} \Big ). 
  $$
  Since $4 \Delta^{\star}$ is divisible by $4$ and $r$ is congruent to $1 \pmod{4}$ both expressions are Dirichlet characters of modulus $4 \Delta^{\star}$ and modulus $r$ respectively. In particular the above expression is equal to
  \begin{equation} \notag
  \chi_{4 \Delta^{\star}}(w) \chi_{r}(w) = \chi_{4 \Delta^{\star}} ( ( 4 v + b)^{\star} r ) \chi_{r} (\Delta^{\star}  \ell ).
  \end{equation}
Using the modularity of $g$, \eqref{eq:transformation}, and combining this with \eqref{equa:one}, \eqref{equa:two}, \eqref{equa:three} we conclude that
  \begin{align*}
    g \Big ( \frac{d}{q} + & \frac{v}{\Delta} + \beta + \frac{i}{X} \Big ) = \overline{\varepsilon_{(4 v + b)^{\star}}} \chi_{4 \Delta^{\star}} ( (4 v + b)^{\star} r ) \chi_{r} ( 4 \Delta^{\star} \ell ) \cdot \Big ( \frac{X}{4 \Delta^{\star} r} \Big )^{k+ \frac 12} \cdot \Big ( \frac{i}{1 - i X \beta} \Big )^{k + \frac 12}
                             \\ & \times g \Big ( - \frac{\overline{w}}{4 \Delta^{\star} r} + \frac{i}{(4 \Delta^{\star} r)^2 \cdot ( \frac{1}{X} - i \beta)} \Big )  .
  \end{align*}
  
  Write $(\overline{a})_{b}$ for the inverse of $a$ modulo $b$.
  Also, let $a_1=(\overline{r \cdot (4 v + b)^{\star}})_{4 \Delta^{\star}}$ and $a_2=(\overline{4 \Delta^{\star} \cdot \ell})_{r}$. Observe
  \[
  (\overline{w})_{4\Delta^{\star}r} \equiv a_1 \pmod{4\Delta^{\star}} \quad \text{ and } \quad (\overline{w})_{4\Delta^{\star}r} \equiv a_2 \pmod{r}.
  \]
  So by the Chinese Remainder Theorem
  $$(\overline{w})_{4 \Delta^{\star} r} \equiv a_1 (\overline{r})_{4 \Delta^{\star}} r + a_2 (\overline{4 \Delta^{\star}})_{r} 4 \Delta^{\star} \pmod{4 \Delta^{\star} r}.$$
  It follows that
  $$
  \frac{\overline{w}}{4 \Delta^{\star} r} \equiv \frac{\overline{r^2 (4 v + b)^{\star}}}{4 \Delta^{\star}} + \frac{\overline{16 (\Delta^{\star})^2 \ell}}{r} \pmod{1} 
  $$
  and since $g(x + i y)$ is $1$-periodic in $x$ the claim follows. \end{proof}

We are now ready to prove the main proposition of this section.

\begin{proof}[Proof of Proposition \ref{scpprop}]
Using the circle method we can re-write the sum on the LHS of \eqref{eq:mainscp} as 
\begin{equation} \notag
\frac{1}{X^{k - \frac 12}} \int_{0}^{1} g \Big ( \alpha + \frac{v}{\Delta} + \frac{i}{X} \Big ) g \Big ( - \alpha + \frac{i}{X} \Big ) e( - \alpha h) d \alpha.
\end{equation}
Let $Q = X^{1/2 + 2 \eta}$ and $\eta \in (0, 1]$ to be determined later (we will choose $\eta = \frac{1}{13}$). Let $1 \leq \Delta \leq X^{\eta / 4} \leq Q^{\eta / 2}$. 
By Lemma \ref{le:circle} the above expression is equal to 
\begin{equation} \label{eq:circled}
\frac{1}{X^{k - \frac 12}} \frac{Q^{2 - \eta}}{2 L}\sum_{q \in \mathcal{Q}} \, \, \sumstar_{d \pamod q} e \Big ( - \frac{h d}{q} \Big ) \int_{- Q^{-2 + \eta}}^{Q^{-2 + \eta}} g \Big ( \frac{d}{q} + \frac{v}{\Delta} + \beta + \frac{i}{X} \Big ) g \Big ( - \frac{d}{q} - \beta + \frac{i}{X} \Big ) e ( - \beta h) d \beta.
\end{equation}
plus an error term of size $\ll X Q^{-\eta/4} \leq X^{1 - \eta / 8}$, where in the estimation of the error term we have used the bound
$
|g(\alpha+i/X)| \ll X^{\frac{k+\frac12}{2}}
$
which holds
uniformly for all $\alpha \in \mathbb R$. This
 follows from the fact that $y^{(k+\frac12)/2}|g(z)|$ is bounded
on $\mathbb H$ since $g$ is a cusp form.

We now write $q = 4 \Delta r$ with $r$ a prime congruent to $1 \pmod{4}$ and just as in Lemma \ref{le:modularity} we write, $d = (4 \Delta) \ell + r b$ with $(\ell, r) = 1$ and $(b, 4 \Delta) = 1$. Then by Lemma \ref{le:modularity} and the triangle inequality, \eqref{eq:circled} is in absolute value less than 
\begin{equation} \label{eq:mainscp2}
 \leq X \cdot \frac{Q^{2 - \eta}}{2 L} \sumstar_{b \pamod{4\Delta}} \sum_{\substack{Q \leq 4 \Delta r \leq 2Q \\ r \equiv 1 \pmod{4} \text{ prime}}} \Big ( \frac{\sqrt{X}}{4 \Delta^{\star} r} \Big )^{k + \frac 12} \cdot \Big ( \frac{\sqrt{X}}{4 \Delta r} \Big )^{k + \frac 12} \cdot S,
\end{equation}
where
\begin{align} \label{eq:Sdef}
  S = 
  \Big | & \sumstar_{\ell \pamod{r}} e \Big ( - \frac{h \ell}{r} \Big ) \int_{-Q^{-2 + \eta}}^{Q^{-2 + \eta}} g \Big ( - \frac{\overline{r^2 (4 v + b)^{\star}}}{4 \Delta^{\star}} - \frac{\overline{16 (\Delta^{\star})^2 \ell}}{r} + \frac{i}{(4 \Delta^{\star} r)^2 (\frac{1}{X} - i \beta)} \Big ) \\ & \times g \Big ( \frac{\overline{r^2 b}}{4 \Delta} + \frac{\overline{16 \Delta^2 \ell}}{r} + \frac{i}{(4 \Delta r)^2 (\frac{1}{X} + i \beta)}  \Big ) \cdot \frac{e(-h \beta) d \beta}{|1 - i \beta X|^{2k + 1}} \Big |. \notag
\end{align}
Notice now that for any $q \geq 1$, 
\begin{equation} \notag
\frac{i}{q^2 (\frac{1}{X} - i \beta)} = \frac{i X}{q^2 \cdot |1 - i \beta X|^2} - \frac{X^2 \beta}{q^2 |1 - i \beta X|^2}.
\end{equation}
Moreover since $|\beta X| \leq Q^{-2 + \eta} X \leq X^{-\eta}$ the imaginary part of the above expression is $(1 + o(1)) X / q^2$. Therefore expanding in Fourier coefficients we can write, for any $\varepsilon > 0$, 
\begin{align} \label{eq:firstexpansiong}
 g\Big ( - \frac{\overline{r^2 (4 v + b)^{\star}}}{4 \Delta^{\star}} & - \frac{\overline{16 (\Delta^{\star})^2 \ell}}{r} + \frac{i}{(4 \Delta^{\star} r)^2 (\frac{1}{X} - i \beta)} \Big ) \\ & = \sum_{n \leq X^{\varepsilon} (4 \Delta^{\star} r)^2 / X} \alpha(n;\beta) n^{\frac{k - \frac 12}{2}} e \Big ( - \frac{n \overline{16 \Delta^{\star} \ell}}{r} \Big ) + O(X^{-A})  \notag
\end{align}
for any $A \ge 1$,
where the coefficients $\alpha(n;\beta)$ are independent of $\ell$ and $|\alpha(n;\beta)| \ll |c(n)| \ll_{\varepsilon} n^{1/2 + \varepsilon}$ for all $\varepsilon > 0$ and $\beta \in \mathbb{R}$ by \eqref{eq:Fourierbd}. Similarly we have for every $\varepsilon > 0$
\begin{align} \label{eq:secondexpansiong}
  g \Big ( \frac{\overline{r^2 b}}{4 \Delta} & + \frac{\overline{16 \Delta^2 \ell}}{r} + \frac{i}{(4 \Delta r)^2 (\frac{1}{X} + i \beta)}  \Big )
  = \sum_{m \leq X^{\varepsilon} (4 \Delta r)^2 / X} \gamma(m;\beta) m^{\frac{k - \frac 12}{2}} e \Big ( \frac{m \overline{16 \Delta^2 \ell}}{r} \Big ) + O(X^{-A})
\end{align}
for any $A \ge 1$,
where the coefficients $\gamma(m;\beta)$ are independent of $\ell$ and $|\gamma(m;\beta)| \ll |c(n)| \ll_{\varepsilon} m^{1/2 + \varepsilon}$ for all $\varepsilon > 0$ and $\beta \in \mathbb{R}$.

Applying  \eqref{eq:firstexpansiong} and \eqref{eq:secondexpansiong}  in \eqref{eq:Sdef}  it follows that
\begin{align*}
S \ll  \int_{-Q^{- 2 + \eta}}^{Q^{-2 + \eta}} \sum_{\substack{m \leq X^{\varepsilon} (4 \Delta r)^2 / X \\ n \leq X^{\varepsilon} (4 \Delta^{\star} r)^2 / X}} \Big | \alpha(n; \beta) \gamma(m; \beta) (m n)^{\frac{k - \frac 12}{2}} S(m \overline{16 \Delta^2}  - n \overline{16 (\Delta^{\star})^2}, -h; r) \Big | d \beta + X^{-100},
\end{align*}
where $S(a,b;c)$ is the standard Kloosterman sum.
Since $0 < |h| < r$ and $r$ is prime the Weil bound shows that the Kloosterman sum above is always $\ll \sqrt{r} \ll \sqrt{ \frac{Q}{\Delta}} $. Therefore, applying Cauchy-Schwarz along with \eqref{eq:Parseval} in the above equation it follows that
\[
\begin{split}
    \ll X^{\varepsilon} \cdot Q^{-2+\eta}  \cdot \left( \frac{(4 \Delta r)^2 }{X}\right)^{\frac{k+\frac32}{2}} \cdot \left( \frac{(4 \Delta r)^2 }{X}\right)^{\frac{k+\frac32}{2}} \cdot \sqrt{ \frac{Q}{\Delta}} .
\end{split}
\]
Also, recall that $Q=X^{\frac12+2\eta}$, $\Delta \le X^{\eta/4}$  and $L \gg \frac{Q^2}{\Delta X^{\varepsilon}}$.
We conclude that \eqref{eq:mainscp2} is
$$
\ll \frac{X^{1+\varepsilon}}{L} \cdot \Delta \cdot \frac{Q}{\Delta}  \cdot \frac{Q^2}{X} \cdot \sqrt{\frac{Q}{\Delta}} \ll X^{\frac34+\frac{25 \eta}{8}+\varepsilon}.
$$
We also recall that when using Lemma \ref{le:circle} in \eqref{eq:circled} we introduced an error of size $X^{1 - \eta / 8}$. Picking $\eta = \frac{1}{13}$ gives a total error of size $\ll X^{1 - \frac{1}{104} + \varepsilon}$ and allows $\Delta$ to go up to $X^{\frac{1}{52}}$.
\end{proof}

\subsection{Proof of Proposition \ref{mainprop}}

We are now ready to prove Proposition \ref{mainprop}. Let $h$ be the indicator function of the integers that can be written as $8 n$ with $n$ odd and square-free. We find that
$$
h(n) = \sum_{\substack{2^{\alpha + 3} d^2  | n \\ d \text{ odd}}} \mu(d) \mu(2^{\alpha}).
$$
Let
$$
h_{\leq Y}(n) = \sum_{\substack{n = 2^{\alpha + 3} d^2 m \\ d \leq Y, \alpha \leq 1 \\ d \text{ odd}}} \mu(d) \mu(2^{\alpha}) \ , \ h_{> Y}(n) = \sum_{\substack{n = 2^{\alpha + 3} d^2 m \\ d > Y, \alpha \leq 1 \\ d \text{ odd}}} \mu(d) \mu(2^{\alpha}). 
$$
By the triangle inequality, 
\begin{align*}
  \sum_{X \leq x \leq 2X} \Big | \sum_{x \leq n \leq x + y} c(n) M((-1)^k n) h_{>Y }(n) \Big | \leq y \sum_{n \leq 4 X} |c(n) M((-1)^k n) h_{> Y}(n)| 
\end{align*}
and by the definition of $h_{> Y}(n)$ this is
\begin{equation} \label{eq:easy}
\ll y \sum_{\substack{d > Y \\ d \text{ odd}}} \sum_{\substack{n \leq 4 X \\ d^2 | n}} |c(n) M((-1)^k n)| .
\end{equation}
A trivial bound gives $M((-1)^k n) \ll_{\varepsilon} M X^{\varepsilon}$.
By Shimura's result \eqref{eq:shimura} and Deligne's bound,
$
|c(d^2 n)| \ll_{\varepsilon} d^{\varepsilon} |c(n)|
$
for all $\varepsilon > 0$. 
Hence, we conclude after an application of Cauchy-Schwarz and \eqref{eq:Parseval} that
$$
\sum_{\substack{n \leq X \\ d^2 | n}} |c(n)| \ll_{\varepsilon} d^{\varepsilon} \sum_{n \leq X / d^2} |c(n)| \ll_{\varepsilon} d^{\varepsilon} \cdot \Big \lfloor\frac{X}{d^2} \Big \rfloor.
$$
Therefore \eqref{eq:easy} is bounded by
$$
\ll_{\varepsilon} y \Big ( \frac{X}{Y} + 1 \Big ) M X^{\varepsilon}.
$$
This contributes $\ll X^{1 - \delta + \varepsilon}$ provided that $Y \geq X^{\delta} M y$. 

Therefore after an application of Cauchy-Schwarz, it remains to obtain a non-trivial upper bound for
$$
\sum_{X \leq x \leq 2 X} \Big | \sum_{x \leq n \le x + y} c(n) M((-1)^k n) h_{\leq Y}(n) \Big |^2 .
$$
We introduce an auxiliary smoothing, and bound the above by
$$
\ll \sum_{x \geq 1} \Big | \sum_{x \leq n \le x + y} c(n) M((-1)^k n) h_{\leq Y}(n) \Big |^2 \exp \Big ( - 4 \pi \frac{x}{X} \Big ) \Big ( \frac{x}{X} \Big )^{k - \frac 12}. 
$$
We note that the implicit constant is allowed to depend on the weight $k + \frac 12$. 
Let $\alpha(n) = c(n) M((-1)^k n) h_{\leq Y}(n)$.
Expanding the square we re-write the above expression as
$$
\frac{1}{X^{k - \frac 12}} \sum_{0 \leq h_1, h_2 \leq y} \sum_{n \ge 1} \alpha(n + h_1) \alpha(n + h_2) n^{k - \frac 12 } \exp \Big ( - \frac{4\pi n}{X} \Big ). 
$$
Grouping terms together and using a Taylor expansion we can re-wite the above as
\begin{equation} \label{eq:allterms}
\begin{split}
&\frac{1}{X^{k - \frac 12}} \sum_{|h| \leq y} \sum_{n \ge 1} \alpha(n) n^{\frac{k - \frac 12}{2}} e^{- \frac{2\pi n}{X}} \alpha(n + h) ( n + h)^{\frac{k - \frac 12}{2}} e^{- \frac{2\pi (n + h)}{X}} \sum_{\substack{0 \le h_1,h_2 \le y \\ h_1-h_2=h}} \left(1+O\left( \frac{y}{n}+\frac{y}{X} \right) \right) \\
&=\frac{1}{X^{k - \frac 12}} \sum_{|h| \leq y} (y + 1 - |h|) \sum_{n \ge 1} \alpha(n) n^{\frac{k - \frac 12}{2}} e^{- \frac{2\pi n}{X}} \alpha(n + h) ( n + h)^{\frac{k - \frac 12}{2}} e^{- \frac{2\pi (n + h)}{X}} +O\left(y^3 M^2 X^{\varepsilon} \right)
\end{split}
\end{equation}
where we used \eqref{eq:Parseval} in the estimation of the error term. The term $h = 0$ contributes
\begin{equation} \label{eq:diagonal}
\ll \frac{y}{X^{k - \frac 12}} \sum_{n \ge 1} |c(n) h_{\leq Y}(n) M((-1)^k n)|^2 \cdot n^{k - \frac 12} e^{- \frac{4\pi n}{X} }. 
\end{equation}
Repeating the same argument as before we see that
$$
y \cdot \frac{1}{X^{k - \frac 12}} \sum_{n \ge 1} |c(n) h_{> Y}(n) M((-1)^k n)|^2 \cdot n^{k - \frac 12} e^{- \frac{4\pi n}{X}} \ll_{\varepsilon}  y \Big  ( \frac{X}{Y} + 1 \Big ) X^{\varepsilon} M^2.
$$
This gives a total contribution of $\ll X^{1 - \delta +\varepsilon}$ provided that $Y > X^{\delta} M^2 y$. 
It follows that \eqref{eq:diagonal} is bounded by
$$
\ll y\cdot \frac{1}{X^{k - \frac 12}} \sum_{\substack{n \text{ odd}}} \mu^2(n) |c(8 n) M((-1)^k 8 n)|^2 \cdot n^{k - \frac 12} e^{-\frac{4\pi n}{X}}+X^{1-\delta+\varepsilon}
$$
as needed. 

We now focus on the terms with $h\neq 0$ in \eqref{eq:allterms}. Opening $h_{\leq Y}$ and $M((-1)^k n)$ we see that the RHS of \eqref{eq:allterms} restricted to $h \neq 0$ is bounded by
\begin{align*}
& \ll_{\varepsilon} y^2 M^2 Y^2 X^{\varepsilon} \cdot \frac{1}{X^{k - \frac 12}}  \\ & \times  \sup_{\substack{0 < |h| \leq y \\ d_1, d_2 \leq Y \\ \alpha_1, \alpha_2 \in \{0, 1\} \\ m_1, m_2 \leq M}} \Big | 
\sum_{\substack{d_1^2 2^{\alpha_1 + 3} | n \\ d_2^2 2^{\alpha_2 + 3} | n + h}}  c(n) \chi_{(-1)^k n}(m_1) n^{\frac{k - \frac 12}{2}} e^{- \frac{2\pi n }{X}} c(n + h) \chi_{(-1)^k (n + h)}(m_2) (n + h)^{\frac{k - \frac 12}{2}} e^{- \frac{2\pi (n + h)}{X}} \Big | .
\end{align*}
By the Chinese
Remainder Theorem the condition $2^{\alpha_1 + 3} d_1^2 | n$ and $2^{\alpha_2 + 3} d_2^2 | n + h$ can be re-written as a single congruence condition to a modulus of size $\le 4 Y^4$. Moreover $\chi_{(-1)^k n}(m_1)$ is $4m_1$-periodic in $n$, where-as $\chi_{(-1)^k (n + h)}(m_2)$ is $4m_2$-periodic in $n$. Therefore fixing the congruence class of $n$ modulo $4 [m_1, m_2]$ fixes the value of $\chi_{(-1)^k n}(m_1) \chi_{(-1)^k (n +h)}(m_2)$. Therefore we can bound the above supremum by
\begin{equation} \label{eq:almostfinal}
\sup_{\substack{0 < |h| \leq y  \\ 1 \leq \Delta \leq 4 Y^4 M^2 \\ \gamma \pmod{\Delta}}} \Big | \sum_{n \equiv \gamma \pmod{\Delta}} c(n) n^{\frac{k - \frac 12}{2}} e^{- \frac{2\pi n}{X} } c(n + h) (n + h)^{\frac{k - \frac 12}{2}} e^{- \frac{2\pi (n + h)}{X}} \Big | .
\end{equation}
Finally, we can write
$$
\mathbf{1}_{n \equiv \gamma \pmod{\Delta}} = \frac{1}{\Delta} \sum_{0 \leq v < \Delta} e \Big ( \frac{v n}{\Delta} \Big ) e \Big ( - \frac{v \gamma}{\Delta} \Big ).  
$$
Plugging this into \eqref{eq:almostfinal} we see that it is
$$
\leq \sup_{\substack{0 < |h| \leq y  \\ 1 \leq \Delta \leq 4 Y^4 M^2 \\ v \pmod{\Delta}}} \Big | \sum_{n \ge 1} c(n) e \Big ( \frac{n v}{\Delta} \Big ) n^{\frac{k - \frac 12}{2}} e^{- \frac{2\pi n}{X} } c(n + h) (n + h)^{\frac{k - \frac 12}{2}} e^{- \frac{2\pi (n + h)}{X}} \Big | 
$$
and without loss of generality we can also assume that $(v, \Delta) = 1$.
It follows that the RHS of \eqref{eq:allterms} restricted to $h \neq 0$ is bounded by
$$
\ll_{\varepsilon} X^{\varepsilon} \cdot \frac{y^2 M^2 Y^2}{X^{k - \frac 12}} \sup_{\substack{0 < |h| \leq y  \\ 1 \leq \Delta \leq 4 Y^4 M^2 \\ (v, \Delta) = 1}} \Big | \sum_{n \ge 1} c(n) e \Big ( \frac{n v}{\Delta} \Big ) n^{\frac{k - \frac 12}{2}} e^{- \frac{2\pi n}{X} } c(n + h) (n + h)^{\frac{k - \frac 12}{2}} e^{- \frac{2\pi (n + h)}{X}} \Big |.
$$

According to Proposition \ref{scpprop} the above expression is $\ll_{\varepsilon} y^2 M^2 Y^2 X^{1 - \frac{1}{104} + \varepsilon}$ provided that $4 Y^4 M^2 \leq X^{\frac{1}{52}}$. Moreover we introduced earlier error terms that were $\ll_{\varepsilon} X^{1 - \delta / 2 + \varepsilon}$ as long as $Y \geq y X^{\delta} M$. Choosing $Y = y X^{\delta} M$ we obtain the restriction $y^4 X^{4 \delta} M^6 \leq X^{\frac{1}{52}}$. We also decide to require that $y^2 M^2 Y^2 X^{1 - \frac{1}{104}} \leq X^{1 - \delta / 2}$ which gives $y^2 M^2 Y^2 X^{\delta / 2} \leq X^{\frac{1}{104}}$ and in particular $y^4 X^{5 \delta / 2} M^4 \leq X^{\frac{1}{104}}$. We posit (somewhat arbitrarily) that we will require $y, M \leq X^{\delta}$. In that case the previous two conditions are verified if $X^{21 \delta / 2} \leq X^{\frac{1}{104}}$ and if $X^{14 \delta} \leq X^{\frac{1}{52}}$. This leads to the choice of $\delta = \frac{1}{1092}$, implying the restriction $y, M \leq X^{\frac{1}{1092}}$ and giving an error term of $\ll X^{1 - \frac{1}{2148}}$. 




\section{Appendix}
\subsection{The structure of the space of half-integral weight forms}
Due to recent work of Baruch-Purkait \cite{BP}, Shimura's correspondence between half-integral weight forms and integral weight forms is better understood. 
Let $A_{k+1/2}^{+}$ denote the \textit{conjugate plus space}, which is defined as $W_4 S_{k+1/2}^+$, where $W_4: S_{k+1/2} \rightarrow S_{k+1/2}$ is given by
\[
(W_4 f)(z)=(-2iz)^{-k-1/2}f\left( \frac{-1}{4z}\right).
\]
Baruch-Purkait proved that $A_{k+1/2}^{+} \bigcap S_{k+1/2}^{+} =\{0\}$. Letting $S_{k+1/2}^{-}$ denote the orthogonal complement of $A_{k+1/2}^{+} \oplus S_{k+1/2}^+$ (w.r.t. the Petersson inner product) they proved that Niwa's isomorphism (see the main theorem in \cite{Niwa}) induces a Hecke algebra isomorphism between $S_{k+1/2}^{-} $ and $S_{2k}^{\text{new}}(2)$ (the space weight $2k$ newforms on $\Gamma_0(2)$). 

Also, let $U_{m}$ be the operator, which acts on power series as follows
\[
U_{m} \left( \sum_{n \ge 1} a_n q^n \right)= \sum_{n \ge 1} a_{mn} q^n.
\]
Niwa proved that $U_4W_4$ is Hermitian on $S_{k+1/2}$ and that $(U_4W_4-\alpha_1)(U_4W_4-\alpha_2)=0$ where $\alpha_1=2^{k} \left( \frac{2}{2k+1}\right)$ and $\alpha_2=-\frac12 \cdot \alpha_1$. The Kohnen plus space $S_{k+1/2}^+$ is the subspace of cusp forms in $S_{k+1/2}$ with $U_4W_4$-eigenvalue equal to $\alpha_1$ (Kohnen \cite[Proposition 2]{KohnenMathAnn}). 




\subsection{Results}
Recall that $\mathbb N^{\flat}=\{ n \in \mathbb N : n=8m \text{ and } \mu^2(2m)=1\}$ and $\mathbb N_g^{\flat}(X)= \{  n \le X : n \in \mathbb N^{\flat} \text{ and } c(n) \neq 0 \}$. Also, let $S_{2k}(N)$ denote the space of weight $2k$ cusp forms on $\Gamma_0(N)$.

\begin{theorem} \label{thm:linearalgebra}
Let $k \ge 2$ be an integer and $g$ be a Hecke cusp form of weight $k+\tfrac12$ on $\Gamma_0(4)$. Then it is possible to normalize $g$ so that its $nth$ Fourier coefficient is real for $n \in \mathbb N^{b}$. 

In addition, suppose that $c(n) \neq 0$ for some $n \in \mathbb N^{\flat}$. Then for every $\varepsilon>0$ the sequence $\{ c(n) \}_{n \in \mathbb N_g^{\flat}(X)}$ has $\gg X^{1-\varepsilon}$ sign changes. Assuming GRH the sequence $\{ c(n) \}_{n \in \mathbb N_g^{\flat}(X)}$ has $\gg  X$ sign changes.

\end{theorem}
\begin{remark}
We will see that if $g$ is of the form
\begin{equation} \label{eq:deggenerate}
g(z)=aG(z)+b (W_4 G)(z)
\end{equation}
with $G \in S_{k+1/2}^+$, and $a,b \in \mathbb C$ with $\frac{a}{b}=- \frac{1}{\alpha_2} \lambda_F(2)$, where $F \in S_{2k}(1)$ is the Shimura lift of $G$, then $c(8n) \mu^2(2n)=0$ for each $n \in \mathbb N$. Otherwise, for a Hecke cusp form $g \in S_{k+1/2}$ not of this form, the proof below and Lemma \ref{le:elementary} imply $c(8n) \mu^2(2n) \neq 0$ for some $n$ so that the conclusion of Theorem \ref{thm:linearalgebra} holds for $g$.

Moreover, for $g$ as in \eqref{eq:deggenerate} with $b \neq 0$ the subsequent argument shows that $c(2n) \mu^2(2n)=\frac{b}{\alpha_2} c_G(8n) \mu^2(2n)$, where $c_G(n)$ denotes the $n$th Fourier coefficient of $G$. Hence in this case we conclude that after suitable normalization, the sequence $\{ c(2n)\mu^2(2n)\}_{n \in \mathbb N}$ has $\gg X$ sign changes as $n$ ranges over integers in $[1,X]$ under GRH and $\gg X^{1-\varepsilon}$ such sign changes unconditionally.
\end{remark}


\begin{proof}
Niwa (see the main theorem of \cite{Niwa}) proved that there is a Hecke algebra isomorphism between $S_{k+1/2}$ and $S_{2k}(2)$.
Denoting Niwa's isomorphism by $\psi$, write $f=\psi(g) \in S_{2k}(2)$.
By Atkin-Lehner \cite[Theorem 5]{AtkinLehner} either $f\in S_{2k}^{\text{new}}(2)$ or  $f$ is an oldform, i.e. $f(z) =\alpha F(z) + \beta F(2z)$ with $F \in S_{2k}(1)$, $\alpha, \beta \in \mathbb C$.

 We first consider the case $\psi(g) =f \in S_{2k}^{\text{new}}(2)$. Here we can apply Shimura's explicit version of Waldspurger's Theorem, which for a fundamental discriminant of the form $(-1)^kd= 8n>0$ with $2n$ square-free gives
\begin{equation} \notag
|c(2n)|^2= L(\tfrac12, f\otimes \chi_{d}) \frac{(k-1)!}{2 \pi^k} \cdot \frac{\langle g,g \rangle}{\langle f, f \rangle}
\end{equation}
(see \cite[Theorem 3B.4]{Shimura2}). The Fourier coefficients can also be normalized so that they are real (this immediately follow from \cite[Theorem 3B.5]{Shimura2}). Also, in this case $c(8n)=\lambda_2 \cdot c(2n)$ where $\lambda_2 \in \{\pm 1\}$ and is independent of $n$ (this follows from \cite[Main Theorem]{Shimura1}, Niwa's isomorphism, and Atkin-Lehner \cite[Theorem 3]{AtkinLehner}, see the proof of Theorem 5 of \cite{BP}). Hence, our argument proceeds just as before and $\{ c(n) \}_{\mathbb N_g^{\flat}(X)}$ has $\gg X^{1-\varepsilon}$ sign changes unconditionally and $\gg X$ under GRH.

Let us now suppose $f \in S_{2k}(2)$ is an oldform. Then by Baruch-Purkait \cite{BP}, there exists $G \in S_{k+1/2}^+$ and $a,b 
\in \mathbb C$ such that
$
g(z)=a \cdot G(z)+b \cdot (W_4 G)(z).
$
If $W_4G =0$ then $g \in S_{k+1/2}^+$, so we are done. Suppose $W_4G \neq 0$, since $S_{k+1/2}^{+} \cap A_{k+1/2}^+=\{0\}$, $W_4G \notin S_{k+1/2}^+$ so $U_4G =U_4 W_4 (W_4 G)=\alpha_2 W_4 G$ since $W_4$ is an involution.
Hence,
\[
g(z)=\sum_{n \ge 1} \left(a \cdot c_G(n) + \frac{b}{\alpha_2} \cdot c_G\left( 4n\right)\right) e(nz).
\]
For $n \in \mathbb N^{\flat}$ write $n=8m$ where $2m$ is square-free then by \eqref{eq:shimura}
\[
c_g(8m)=a \cdot c_G(8m)+ \frac{b}{\alpha_2} \cdot c_G(2^2 \cdot 8m)=c_G(8m) \cdot \left(a+\frac{b}{\alpha_2} \cdot \lambda_F(2) \right),
\]
where $\lambda_F(\cdot)$ denotes the Hecke eigenvalues of the level $1$ modular form which corresponds to $G$ under $\psi$.
By assumption $a+\frac{b}{\alpha_2} \cdot \lambda_F(2)\neq 0$, since otherwise $c_g(8m)\mu^2(2m)= 0$ for every $m \in \mathbb N$. Moreover, since $G \in S_{k+1/2}^+$ we know $\{c_G(n)\}_{n \in \mathbb N_G^{\flat}(X)}$ has $\gg X$ sign changes under GRH and $\gg X^{1-\varepsilon}$ unconditionally, so the result follows.  \end{proof}

Additionally, it is possible to extend our result to level $4N$ with $N$ square-free and odd if we also restrict to fundamental discriminants that lie in a suitable progression. For each prime divisor of $p|N$ we require that
$
\chi_{d}(p)=w_{p}
$
where $w_{p}$ is the eigenvalue of the Atkin-Lehner operator $W_{p}$. So by the Chinese Remainder Theorem there exists $\eta \pmod{ N}$ such that for $d \equiv \eta \pmod{N}$ we have $\chi_d(p)=w_p$. Let $\mathbb N_{N}^{\flat}=\{ n \in \mathbb N : n=8m, \mu^2(2m)=1,  \text{ and } (-1)^kn \equiv \eta \pmod{N}\}$ and $\mathbb N_{N,g}^{\flat}(X)=\{ n \le X : n \in \mathbb N_{N}^{\flat} \text{ and } c(n) \neq 0 \}$. Note that for $n \in \mathbb N_N^{\flat}$ by construction $(N,n)=1$. 

Let $S_{k+1/2}(4N)$ denote the space of weight $k+1/2$ cusp forms of level $4N$, with $N$ odd and square-free. Also, let $S_{k+1/2}^{-}(4N)$ be as defined by Baruch-Purkait (see Section 6.3 of \cite{BP}), who showed that this space is isomorphic to $S_{2k}^{\text{new}}(2N)$. This complements Kohnen's result \cite{KohnenNewform} that $S_{k+1/2}^{\text{new}}(4N)$ is isomorphic to $S_{2k}^{\text{new}}(N)$.

We can also prove the following result.

\begin{theorem}
Let $k \ge 2$ and $N$ be odd and square-free.  Suppose $g \in S_{k+1/2}^{+, \text{new}}(4N)$ or $g \in S_{k+1/2}^{-}(4N)$, is a Hecke cusp form.
Then it is possible to normalize $g$ so that its $nth$ Fourier coefficient is real for $n \in \mathbb N_{N}^{\flat}$. Moreover, for every $\varepsilon>0$ the sequence $\{ c(n) \}_{n \in \mathbb N_{N,g}^{\flat}(X)}$ has $\gg X^{1-\varepsilon}$ sign changes. Assuming GRH the sequence $\{ c(n) \}_{n \in \mathbb N_{N,g}^{\flat}(X)}$ has $\gg  X$ sign changes.
\end{theorem}

\begin{proof} We will only sketch the argument, since our main propositions need to be modified when $N  \neq 1$. 
First, suppose 
$g \in S_{k+1/2}^{-}(4N)$.
Then by \cite[Theorem 8]{BP} multiplicity one holds in $S_{k+1/2}^{-}(4N)$ in the whole space $S_{k+1/2}(4N)$\footnote{Multiplicity one also holds in the space $S_{k+1/2}^+$, but fails in the entire space $S_{k+1/2}$ (consider $g \in S_{k+1/2}^+$ and $W_4 g$).}. Hence for $(-1)^kd \in \mathbb N_{N}^{\flat}$ we can apply Shimura's result \cite[Theorem 3B.4]{Shimura2} to get that
\begin{equation} \label{eq:shimurawald}
|c(2n)|^2=2^{\omega(N)} \frac{(k-1)!}{2\pi^k}\frac{\langle g,g \rangle}{\langle f,f \rangle} L(\tfrac12, f \otimes \chi_d)
\end{equation}
where $\omega(N)=\sum_{p|N} 1$. Here we have used the condition $d \equiv \eta \pmod{N}$ to estimate the Euler product present in the statement of the theorem along with \cite[Theorem 3]{AtkinLehner}. Using \cite[Theorem 3B.5]{Shimura2} it also follows that the Fourier coefficients of $g$ can be normalized so that they are real. Also, in this case just as before we have $c(8n)=\lambda_2 \cdot c(2n)$ with $\lambda_2 \in \{ \pm 1\}$.

Next, suppose $g(z) \in S_{k+1/2}^{+, \text{new}}(4N)$.
Kohnen \cite{KohnenNewform} proved $S_{k+1/2}^{+,\text{new}}(4N)$ and $S_{2k}^{\text{new}}(N)$ are isomorphic as Hecke algebras. In this setting we can apply Proposition 4.2 of Kumar and Purkait \cite{KP} and it follows $g \in S_{k+1/2}^{+,\text{new}}(4N)$ can be normalized so that it has real (and algebraic) Fourier coefficients.  Moreover, for discriminants $(-1)^k d \in \mathbb N_N^{\flat}$ Corollary 1 of \cite{KohnenWald} implies
\begin{equation} \label{eq:kohnenwald}
|c(8n)|^2=2^{\omega(N)} \frac{\Gamma(k)}{\pi^k}\frac{\langle g,g \rangle}{\langle f,f \rangle} L(\tfrac12, f \otimes \chi_d).
\end{equation}

 Using the formulas \eqref{eq:shimurawald} and \eqref{eq:kohnenwald} we have in each case above that
\[
\sum_{n \in \mathbb N_{g,N}^{\flat}(X) } |c(8n)|^4 M((-1)^k 8n; \tfrac12)^4 \ll \sumfund_{|d| \le X}  L(\tfrac12, f\otimes \chi_d) M(d; \tfrac12)^4.
\]
To bound this mollified moment, the only modification needed in the proof of Proposition \ref{prop:moments} is to change the definition of $I_0$ so that $I_0=(c, X^{\theta_0}]$ with $c$ sufficiently large in terms of $N$. Repeating the argument (with no further modifications) we arrive at
\[
 \sumfund_{|d| \le X}  L(\tfrac12, f\otimes \chi_d) M(d; \tfrac12)^4 \ll X.
\]
We also need to prove
\[
\sum_{n \in \mathbb N_{g,N}^{\flat}(X) } |c(8n)|^2 M((-1)^k 8n; \tfrac12)^2 \asymp X.
\]
This follows in the same way as before once we have established an analog of Proposition \ref{prop:twisted} for $f \in S_{2k}^{\text{new}}(M)$ with $M=N$ or $M=2N$, where we average over discriminants $(-1)^kd \in \mathbb N_N^{\flat}$. The necessary modifications for this computation have already been worked out in the paper of Radziwi\l \l\, and Soundararajan \cite{RadziwillSound}. 
Finally, we need to establish the estimate
\[
\sum_{X \leq x \leq 2X} \Big |  \sum_{\substack{x \leq 8n \leq x + y \\ n \text{ odd} \\ (-1)^k 8n \equiv \gamma \pamod{N} }} \mu^2(n) c(8n) M((-1)^k 8 n) \Big | \ll X \sqrt{y} + X^{1-\frac{1}{2148}+\varepsilon}.
\]
To do this we first need to modify the proof of Lemma \ref{le:modularity} in a straightforward way. From here we arrive at the analog of Proposition \ref{scpprop}, for any $g \in S_{k+1/2}(4N)$.  To establish the above bound we repeat the argument used in the proof of Proposition \ref{mainprop}.  The only modification necessary is that the range of $\Delta$ in \eqref{eq:almostfinal} will now be $1 \le \Delta \le 16  N Y^4 M^2$ to account for the progression $(-1)^k 8n \equiv \eta \pmod{N}$. 

Combining the three estimates above we argue as in the proof of Theorem \ref{thm:sign} thereby finishing the proof. \end{proof}
\section{Acknowledgements}
We would like to thank Abhishek Saha for many helpful conversations regarding Kohnen's plus space and for pointing out Shimura's paper \cite{Shimura2}. We would also like to express our gratitude to Dinakar Ramakrishnan who provided us with comments on an earlier draft of the appendix and to Philippe Michel for useful suggestions regarding the shifted convolution problem.

\bibliographystyle{amsplain}
\bibliography{references}
\end{document}